\documentclass[a4paper,12pt]{amsart}

\def\vacio{\small\text{\O}}

\makeatletter
\@namedef{subjclassname@2020}{%
  \textup{2020} Mathematics Subject Classification}
\makeatother

\usepackage{latexsym}
\usepackage{amssymb}
\usepackage{graphicx}
\usepackage{wrapfig}
\usepackage{amsthm,accents,mathtools,abraces}
\usepackage{float}
\usepackage{epsfig}
\usepackage{tikz}
\usetikzlibrary{cd}
\usetikzlibrary{arrows}
\usetikzlibrary{arrows.meta}
\usetikzlibrary{positioning,automata}
\usepackage{multirow}
\usepackage[toc,page]{appendix}
\usepackage[footnotesize]{caption}
\RequirePackage{color}

\usepackage{amscd,enumerate}
\usepackage{amsfonts}

\input xy
\xyoption{all}

\usepackage{dutchcal}

\usepackage{hyperref}
\hypersetup{ colorlinks,
linkcolor=blue,
filecolor=green,
urlcolor=blue,
citecolor=blue }

\newtheorem{lemma}{Lemma}[section]
\newtheorem{corollary}[lemma]{Corollary}
\newtheorem{theorem}[lemma]{Theorem}
\newtheorem{proposition}[lemma]{Proposition}
\newtheorem{remark}[lemma]{Remark}
\newtheorem{definition}[lemma]{Definition}

\newtheorem{example}[lemma]{Example}

\newtheorem{notation}[lemma]{Notation}

\newcommand{\N}{{\mathbb{N}}}

\newcommand{\h}{{\mathcal{H}}}
\def\a{\alpha}
\def\b{\beta}

\def\l{\lambda}

\def\ideal{\mathop{\hbox{\rm ideal}}}
\def\c1{\subseteq^1}
\def\op#1{#1^{\text{\rm op}}}
\newcommand{\uloopr}[1]{\ar@'{@+{[0,0]+(-4,5)}@+{[0,0]+(0,10)}@+{[0,0] +(4,5)}}^{#1}}

\newcommand{\ann}{{\rm{Ann}}}
\def\Pb{\mathop{\hbox{\rm P}_{b^\infty}}}
\def\Pl{\mathop{\hbox{\rm P}_{l}}}
\def\Pc{\mathop{\hbox{\rm P}_{\hbox{\rm\tiny c}}}}
\def\Pec{\mathop{\hbox{\rm P}_{ec}}}
\def\Plce{\mathop{\hbox{\rm P}_{lce}}}
\def\Pbp{\mathop{\hbox{\rm P}_{b_p^{\infty}}}}
\def\soc{\mathop{\hbox{\rm Soc}}}
\def\socs{\mathop{\hbox{\rm \tiny Soc}}}
\def\P{\mathop{\mathcal P}}
\def\top{\mathop{\hbox{\sffamily{\small \bf DCC}}}}%
\def\tops{\mathop{\hbox{\sffamily{\tiny \bf DCC}}}}
\def\ext{\mathop{\hbox{\rm ext}}}
\def\path{\mathop{\hbox{\rm Path}}}
\def\reg{\mathop{\hbox{\rm Reg}}}
\def\s{\mathop{\hat{s}}}
\def\H{\mathop{\mathcal H}}
\def\remove#1{}
\definecolor{turquoise2}{rgb}{0,0.898039,0.933333}
\definecolor{magenta}{rgb}{1,0,1}
\definecolor{olivedrab}{rgb}{0.419608,0.556863,0.137255}
\definecolor{purple2}{rgb}{0.568627,0.172549,0.933333}
\definecolor{amethyst}{rgb}{0.6, 0.4, 0.8}
\definecolor{ao(english)}{rgb}{0.0, 0.5, 0.0}
\definecolor{atomictangerine}{rgb}{1.0, 0.6, 0.4}
\definecolor{amber(sae/ece)}{rgb}{1.0, 0.49, 0.0}
\definecolor{alizarin}{rgb}{0.82, 0.1, 0.26}
\definecolor{auburn}{rgb}{0.43, 0.21, 0.1}
\definecolor{aqua}{rgb}{0.0, 1.0, 1.0}

\title{Invariant ideals in Leavitt path algebras}
\author[C. Gil Canto]{Crist\'obal Gil Canto}\address{Departamento de Matemática Aplicada, Universidad de M\'alaga, Espa\~na.}
\email{cgilc@uma.es}

\author[D. Mart\'{\i}n Barquero]{Dolores Mart\'{\i}n Barquero}
\address{Departamento de Matem\'atica Aplicada, Universidad de M\'alaga, Espa\~na. }

\email{dmartin@uma.es}

\author[C. Mart\'{\i}n Gonz\'alez]{C\'andido Mart\'{\i}n Gonz\'alez}
\address{Departamento de \'Algebra, Geometr\'{\i}a y Topolog\'{\i}a, Universidad de M\'alaga, Espa\~na.}

\email{candido\_m@uma.es}

\thanks{ The  authors are supported by the Junta de Andaluc\'{\i}a  through projects  FQM-336 and UMA18-FEDERJA-119 and  by the Spanish Ministerio de Ciencia e Innovaci\'on   through project  PID2019-104236GB-I00,  all of them with FEDER funds.}
\begin{document}

\maketitle
\begin{abstract}
It is known that the ideals of a Leavitt path algebra $L_K(E)$ generated by $\Pl(E)$, by $\Pc(E)$ or by $\Pec(E)$ are invariant under isomorphism. Though the ideal generated by $\Pb(E)$ is not invariant we find its \lq\lq natural\rq\rq\ replacement (which is indeed invariant): the one generated by the vertices of $\Pbp$ (vertices  with  pure  infinite  bifurcations). We also give some procedures to construct invariant ideals from previous known invariant ideals. One of these procedures involves topology, so we introduce the $\tops$ topology and relate it to annihilators in the algebraic counterpart of the work. To be more explicit: if $H$ is a hereditary saturated subset of vertices providing an invariant ideal, its exterior $\ext(H)$ in the
$\tops$ topology of $E^0$ generates a new invariant ideal. The other constructor of invariant ideals
is more categorical in nature. Some hereditary sets can be seen as functors from graphs to sets (for instance $\Pl$, etc). Thus a second method emerges from  the possibility of applying the induced functor to the quotient graph. The easiest example is the known socle chain $\soc^{(1)}(\ )\subset\soc^{(2)}(\ )\subset\cdots$ all of which are proved to be invariant. We generalize this idea to any hereditary and saturated invariant functor. Finally we investigate a kind of composition of hereditary and saturated functors which is associative.
\end{abstract}

\section{Introduction and preliminaries}

As well-known examples of Leavitt path algebras arise the so-called primary colours. They respectively correspond to the ideal of $L_K(E)$ generated by the set of line points $\Pl$, the ideal generated by the vertices that lie on cycles without exits $\Pc$ and the one generated by the set in extreme cycles $\Pec$. Those sets constitute an essential ingredient in the structure of Leavitt path algebras. Firstly, the ideal generated by $\Pl$ is precisely the socle of $L_K(E)$ \cite{AMMS1, AMMS2}. What is more, in \cite{CGKS} the ideal generated by $\Pl$ has recently been proved to be the largest locally left/right artinian ideal inside $L_K(E)$ and respectively, $I(\Pc)$ the largest locally left/right noetherian without minimal idempotents. On the other hand, the ideal generated by $\Pec$ is purely infinite \cite{CGKS,CMMS}. Another important fact is that all of them have been proved to be invariant ideals under rings isomorphisms for Leavitt path algebras: the ideal generated by $\Pl$ in \cite{AMMS1}, $I(\Pc)$ in \cite{ABS} and the ideal generated by $\Pec$ in \cite{CGKS}. 

For an arbitrary graph $E$, the union of the three sets above mentioned that give us the primary colours, together with the set $\Pb$ generate an ideal of $L_K(E)$ which is dense \cite{CGKS}. In this work we give a step forward in studying the invariance of this another key piece of $L_K(E)$. Although in general we see that the ideal generated by $\Pb$ is not invariant, we will determine a subset of vertices inside $\Pb$ in which the answer is positive: the set $\Pbp$ of vertices with pure infinite bifurcations.  Furthermore, the main goal of this paper is to develop a machinery that produces invariant ideals for Leavitt path algebras. In order to do that, we introduce a topology in the set of vertices of a graph $E$ that we will call $\top$ topology. Basically the closed sets of this topology will be the set of vertices that connects to the given set. On the one hand, we will establish graph-theoretic notions for Leavitt path algebras in topological terms. On the other hand, via category theory, we will think of the saturated and hereditary set of a graph as an operator (actually a functor). Roughly speaking, we prove that if $H$ is a hereditary and saturated invariant functor, then the functor associated to the set of vertices which do not connect with $H$ is also invariant. 
Using these tools, we will prove that the ideal generated by the subset of vertices of $\Pb$ which do not connect to $\Pl \cup \Pc \cup \Pec$ is invariant (the so-called set $\Pbp$). In addition, given a hereditary and saturated functor $H$, we will construct a chain of hereditary and saturated functors $H=H^{(1)}\subset H^{(2)}\subset\cdots\subset H^{(n)}\subset\cdots $
with each $H^{(i)}$, $i\ge 1$ being hereditary and saturated and such that $H^{(i)}$ is invariant when $H$ is.

As an extra motivation we would like to link in some future development, the discovery of invariant ideals to the \lq\lq rigidity\rq\rq\ of the automorphism group of a Leavitt path algebra. If the amount of invariant ideals is
large enough the freedom degrees of automorphisms are under control. As a general rule, we can guess that the number of invariant ideals is directly proportional to the rigidity of the automorphism group.
\vskip 0.2cm

This paper is organized as follows. In Section \ref{dcc} we introduce the $\top$ topology of the set of vertices of a graph. 
This topological setting is the counterpart of the algebraic one provided by annihilators. 
The motivation that guide us, is that the exterior (in the $\top$ topology) of an invariant hereditary and saturated set $H$ is again invariant. 
In Subsection \ref{cati} we use some tools borrowed from the theory of categories and functors which are substantial for our work.  We prove that the ideal generated by $\Pb$ is not invariant (we also check that the ideal generated by $\Pec\cup\Pb$ is not invariant). In Section \ref{anni} we use annihilators to produce invariant ideals. In terms of functors 
the main point here is that for every hereditary and saturated invariant functor $H$, its exterior  
$\ext(H)$ is again invariant. Theorem \ref{enjundia} proves that $\Pbp$ is invariant. Then, in Section \ref{rosa} we find how to construct series of functors $H=H^{(1)}\subset\cdots\subset H^{(i)}\subset H^{(i+1)}$ which are invariant when $H$ is. As a motivation we apply this construction to $\Pl$ in Section \ref{soclechain} and we characterize graphically those functors $\Pl^{(n)}$ (Theorem \ref{encasa}). Similarly the same idea holds also to $\Pc$ (Theorem \ref{enbarco2}) and the possibilities are broad enough because we can also perform a kind of composition $H_1*H_2$ in Subsection \ref{colada} which turns out to be associative (Theorem \ref{associative}).
Thus, restricting the universe conveniently it appears to be possible to construct a monoid of 
isomorphism classes of invariant hereditary and saturated functors.

We briefly recall concepts which will be used throughout the paper. The basic definitions on graphs and Leavitt path algebras can be seen in the book \cite{AAS}.

Let $E=(E^0,E^1,r_E,s_E)$ be a \emph{directed graph}. Then we define $\op{E}$ as the graph $\op{E}=((\op{E})^0,(\op{E})^1,r_{\op{E}},s_{\op{E}})$ where $(\op{E})^0=E^0$, $(\op{E})^1=E^1$, $s_{\op{E}}=r_E$ and $r_{\op{E}}=s_E$. 
As usual, we will drop the subscript of the source and target maps when no possible ambiguity arises. We denote by $\hat{E}$  the \emph{extended graph} of $E$; concretely, $\hat{E}=(E^0,E^1 \cup (E^1)^*,r',s')$ where $(E^1)^*=\{e^* \; \vert \; e \in E^1\}$, ${r'|}_{E^1}=r$, ${s'|}_{E^1}=s$, $r'(e^*)=s(e)$ and $s'(e^*)=r(e)$ for all $e \in E^1$. A graph $E$ is {\it row-finite} if $ s^{-1}(v)=\{e \in E^1 \; \vert \; s(e)=v\}$ is a finite set for all $v \in E^0$. In this article we will consider row-finite graphs unless otherwise specified. The set of regular vertices (those which are neither sinks nor infinite emitters) is denoted by ${\rm Reg}(E^0)$.  The set of all paths of a graph $E$ is denoted by ${\rm Path}(E)$. If there is a path from a vertex $u$ to a vertex $v$, we write $u\geq_{E} v$ and if $v \in H \subset E^{0}$, we write $u\geq_{E} H$ (we eliminate the subscript $E$ in case there is no ambiguity about the graph). 
A subset $H$ of $E^{0}$ is called \textit{hereditary} if, whenever $v\in H$ and
$w\in E^{0}$ satisfy $v\geq w$, then $w\in H$. A  set $X$ is
\textit{saturated} if for any  vertex $v$ which is neither a sink nor an infinite emitter, $r(s^{-1}(v))\subseteq X$
implies $v\in X$. We will denote the subset of all subsets of $E^0$ which are hereditary and saturated  by $\h_E$. Given a nonempty subset $X$ of vertices, we define the \emph{tree} of $X$, denoted by $T_E(X)$, as the set
$$T_E(X):=\{u\in E^0 \ \vert \  x\geq u \ \text{for some} \ x\in X\}.$$ When there is no possible confusion we denote by $T(X)$. 
This is a hereditary subset of $E^0$.  The notation $\overline{X}$ ($\overline{X}^E$ if we want to emphasize the graph $E$) will be used for the hereditary and saturated closure of a non empty set $X$, which is built, for example, in \cite[Lemma 2.0.7]{AAS} in the following way: Let $\Lambda^0(X) := T(X)$ and
 \begin{equation}\label{onion}
  \Lambda^{n+1}(X):=\{v \in \text{Reg}(E^0): r(s^{-1}(v))\in \Lambda^n(X)\}\cup \Lambda^n(X).   
 \end{equation} Then $\overline{X}=\cup_{n\geq 0} \Lambda^n(X)$. If there is no confusion with respect to the set $X$ we are considering, we simply write $\Lambda^{n}$.
A  vertex $u$ in a graph $E$ is a {\it bifurcation, or there is a bifurcation at $u$} if $s^{-1}(u)$ has at least two elements. A vertex $v$ in a graph $E$ will be called a {\it line point} if there are neither bifurcations nor cycles at any vertex $w \in T(v)$. We will denote by $\Pl{(E)}$ the set of all line points of $E^0$.
An {\it exit} for a path $\mu = e_1 \ldots e_n$ with $n \in \mathbb{N}$, is an edge $e $ such that $s(e)=s(e_i)$ for some $i$ and $e \ne e_i$. We say that $E$ satisfies \emph{Condition} (L) if every cycle in $E$ has an exit. We denote by $\Pc(E)$ the set of vertices of the graph lying in cycles without exits.
A cycle $c$ in a graph $E$ is an {\it extreme cycle} if $c$ has exits and for every $\l\in \hbox{Path}(E)$ starting in a vertex in $c^0$ there exists $\mu \in \hbox{Path}(E)$  such that $0\ne\l\mu$ and $r(\l\mu)\in c^0$. We will denote by $\Pec(E)$ the set of vertices which belong to extreme cycles. 
Besides, the set of all vertices $v \in E^0$ whose tree $T(v)$ contains infinitely many bifurcation vertices or at least one infinite emitter is denoted by $\Pb{(E)}$. Again we will eliminate $E$ in these sets if there is no ambiguity about the graph we are considering.
If $H$ is a hereditary subset of $E^0$, then we can define $I(H)$ as in \cite[Lemma 2.4.1]{AAS}. The set of natural numbers (included $0$) will be denoted by $\N$. For a given set $X$, we will denote by $\P(X)$ the power set of $X$. 
In an algebra $A$ the ideal generated by an element $z\in A$ will be denoted $\ideal(z)$.

\section{A graph topology}\label{dcc}

In this section we define the $\top$ topology  which has sense in any graph. We prove some results relating the topology with algebraic properties of the associated Leavitt path algebra. Since density in the $\top$
implies density of the related ideal we have introduced the term \lq\lq connection\rq\rq\ within the topology name. We see that certain properties of subsets of $E^0$, for instance, \lq\lq being hereditary\rq\rq\ is a topological property, so this property is preserved under homeomorphisms. A corollary of the existence of a homemorphism between two graphs is the preservation of certain elements (for instance the cardinal of the initial and of the terminal set of vertices). In certain type of graphs this induces a preservation of the number of sinks and/or sources (see Remark \ref{pollo}). \remove{The optimum manifestation of this conservation laws, is when the graph is acyclic and satisfies Condition \hbox{SING}. In this case, homeomorphisms between the set of vertices of two graphs induce isomorphism between the lattices of ideals of the corresponding Leavitt path algebras (paragraph just before \eqref{eqiv}).} \newline
\indent However the main reason to consider this topology is (roughly speaking) that when an ideal $I(H)$ is invariant under isomorphism, then $I(\ext(H))$ is also invariant (here $\ext(H)$ is the exterior in the $\top$ topology of $H$). This is proved in Proposition \ref{dormit}(\ref{leo}). 
Thus the exterior operation $H\mapsto\ext(H)$ is one of the relevant tools in the construction of invariant ideals. We prove that the shift process induces a $\top$ continuous map between the set of vertices (Theorem \ref{shiftcont}).

\begin{definition}{\rm  Let $E$ be a graph and define $c\colon\P(E^0)\to\P(E^0)$ the map such that for any $A\subset E^0$ we have $c(A):=\{v\in E^0 \; \vert \; v\ge A\}$. Then this map defines a Kuratowski closure operator that is: 
\begin{itemize}
\item[(i)] $c(\vacio)=\vacio$; 
\item[(ii)] $A\subset c(A)$; 
\item[(iii)] $c(A)=c(c(A))$, and 
\item[(iv)] $c(A\cup B)=c(A)\cup c(B)$, for any $A,B\subset E^0$. 
\end{itemize}
Consequently there is a topology in $E^0$ whose closed sets are those $A\subset E^0$ such that $A=c(A)$ (see \cite[Chapter III, Section 5, Theorem 5.1]{DU}). We will call this the $\top$ {\it topology} (for $\top=\hbox{Directed Connection Closure}$).
}
\end{definition}

The open sets are $E^0\setminus c(A)$ for $A\subset E^0$, so an open set $O$ is one for which there is a subset $A\subset E^0$ such that $O=\{v\in E^0\colon v\not\ge A\}$. We will use the notation $A'$ for the set of all vertices $v$ such that $v\not\ge A$ (in fact $A'$ is the exterior of $A$ in the $\top$ topology as we will see later). 

\begin{remark}\rm
For a hereditary subset $H\subset E^0$ we always have 
$\overline{H}\subset c(H)$ (see \cite[Lemma 1.2]{CMMSS}). In particular, if $\overline{H}=E^0$, the set $H$ is dense in the $\top$ topology. Furthermore, the characterization of the (topological) density of a hereditary set $H$ can be given in terms of the density of the ideal $I(H)$.
\end{remark}
\begin{lemma}\label{romeo}
Let $H$ be a hereditary subset $H\subset E^0$.  Then $H$ is dense in the $\top$ topology if and only if the ideal $I(H)$ is dense in $L_K(E)$.  
\end{lemma}
\begin{proof}
 If $c(H)=E^0$ then for any $v\in E^0$ we have $v\ge H$ hence applying \cite[Proposition 1.10]{CMMSS} the ideal $I(H)$ is dense in
$L_K(E)$. Reciprocally if $I(H)$ is dense, then any vertex $v$ connects to $H$ hence $E^0=c(H)$.
\end{proof}

Some graph properties are invariant under graph homeomorphisms. We list some of them in the following propositions. Recall that a \emph{clopen} set in a topology is a set which is both open and closed.  We will use also the notation $A^c$ for the complementary of $A$ (if the ambient universe is clear).

\begin{proposition}
Let $E$ be an arbitrary graph. Then $E$ is a connected graph if and only if $E$ is connected in the sense of the $\top$ topology.
\end{proposition}

\begin{proof}
First suppose that $E$ is a connected graph. Let $A$ be a subset of $E^0$ which is clopen in the $\top$ topology. We will see that $A=\vacio$ or $A=E^0$. Suppose $A \neq \vacio$ and take $v \in A$. Now for every $w \in E^0$ we have two possibilities: $w \ge v$ or $v \ge w$. If $w \ge v$ then $w \ge A$ and therefore $w \in c(A)=A$, so $w \in A$. In the second case $v \ge w$, if we had  $w \notin A$, then $v \ge A^c$ hence $v\in c(A^c)=A^c$ and so $v \in A^c$, that is, $v \notin A$ which is a contradiction. In short $w \in E^0$ implies $w \in A$, so $E^0 = A$.

For the converse suppose $E$ is connected in the sense of the $\top$ topology. On the contrary, assume $E$ is not a connected graph: $E^0= {\bigsqcup}_{i\in I} E_i^0$ with each $E_i$ a connected graph. We claim that every $E_i^0$ is closed because $c(E_i^0)=\{ v \in E^0 \; \vert \; v \ge E_i^0\} \subseteq E_i^0 \subseteq c(E_i^0)$. Also $E_i^0$ is open. In order to prove this claim, write $E_i^0 = (\bigsqcup_{j \neq i} E_j^0)^c$. Now $\bigsqcup_{j \neq i} E_j^0$ is closed since $v \ge \bigsqcup_{j\ne i} E_j^0$ implies $v \in \bigsqcup_{j\ne i} E_j^0$. To sum up, $E_i^0$ is clopen and since $E$ is connected, in the sense of the $\top$ topology, then there exists an unique $i$ such that $E_i^0 \neq \vacio$ and $E^0=E_i^0$.
\end{proof}

Some purely graph-theoretic notions can be formalized in topological terms.
Remind that for a subset $S$ of a topological space $X$, the exterior of $S$, denoted $\ext(S)$ is the complementary of the closure of $S$, that is, $\ext(S)=c(S)^c$. 

\begin{proposition}
A  subset $H\subset E^0$ is hereditary if and only if $H=\cap_{w\notin H}\ext(w)$.
\end{proposition}
\begin{proof} Assume that $H$ is hereditary. If $v\in H$ and $w\notin H$ then 
$v\not\ge w$ hence $v\in\ext(w)$. So
$H\subset\cap_{w\notin H}\ext(w)$. On the other hand, if $v\in \cap_{w\notin H}\ext(w)$
and $v\notin H$ we have $v\in\ext(v)$ which is a contradiction. So far we have proved that if $H$ is hereditary, then the equality $H=\cap_{w\notin H}\ext(w)$ holds. Conversely, if $H=\cap_{w\notin H}\ext(w)$ then $H$ is hereditary. Indeed, taking $v\in H$ and $v\ge w$, if $w\notin H$ then $v\in\ext(w)$ which means $v\not\ge w$ which is a contradiction.
\end{proof}
Consequently being hereditary is a topological property. Since the intersection of a  family of hereditary subsets is a hereditary subset, also the hereditary closure of a subset $X\subset E^0$ is a topological construction: the intersection of all the hereditary subsets containing $X$.   
\medskip

Recall that a subset $H\subset E^0$ is said to be saturated if $\forall v\in \reg(E^0)$, where $r(s^{-1}(v))\subset H$ this implies $v\in H$. We can extend the source function $s\colon E^1\to E^0$ to a function ${\s}\colon\path(E)\setminus E^0\to E^0$ where ${\s}(\lambda)=s(f_1)$ for $\lambda=f_1\cdots f_n\in\path(E)$. When $H$ is hereditary the following are equivalent:
\begin{equation}\label{odinurg}
r(s^{-1}(v))\subset H \Leftrightarrow r({\hat {s}}^{-1}(v))\subset H.
\end{equation}
Indeed $s^{-1}(v)\subset{\s}^{-1}(v)$ hence $r(s^{-1}(v))\subset r({\s}^{-1}(v))$, implying the right  to left implication. Now, if $r(s^{-1}(v))\subset H$ and $\l\in{\s}^{-1}(v)$ then writing $\l=f_1\cdots f_n$ we have 
$r(f_1)\in H$ hence $H$ being hereditary implies $r(\l)\in H$. Thus 
$r({\s}^{-1}(v))\subset H$. The equivalence given in \eqref{odinurg} allows a reformulation of the definition of saturated hereditary subset. A hereditary subset $H$ is saturated if for any regular vertex $v$ one has the implication
\begin{equation}\label{bbva}
r({\s}^{-1}(v))\subset H \Leftrightarrow v\in H.
\end{equation}

Being saturated is not a topological construction as the following example shows.

\begin{example}\label{nublado}\rm\par
In the graph $E$ of Figure \ref{soleado}, 

\begin{figure}[ht]
\hbox{\hskip 3cm
\begin{tikzpicture}[->]
\draw[fill=black] (-0.1,0) circle (1pt);
\draw[fill=black] (1.1,0.85) circle (1pt);
\draw[fill=black] (1.3,-.5) circle (1pt);
\node at (-1.3,0){$E$:};
\node at (-0.4,-0.1) {\tiny $v$};
\node at (1.5,-0.4) {\tiny $w$};
\node at (1.3,1){\tiny $u$};
\node at (0.4,0.6) {\tiny $f$};
\node at (0.4,-0.5) {\tiny $g$};
\draw [->] (-0.15,-0.1) arc (360:40:9pt);
\draw[] (0,0.1) -> (1,0.8);
\draw[] (0,-0.1) -> (1.2,-.5);
\end{tikzpicture}
\hskip 1cm
\begin{tikzpicture}[->]
\draw[fill=black] (-0.1,0) circle (1pt);
\draw[fill=black] (1.1,0.85) circle (1pt);
\draw[fill=black] (1.3,-.6) circle (1pt);
\node at (-1,0){$F$:};
\node at (-0.4,0) {\tiny $v'$};
\node at (1.55,-0.4) {\tiny $w'$};
\node at (1.3,1){\tiny $u'$};
\node at (0.4,0.6) {\tiny $f'$};
\node at (0.4,-0.6) {\tiny $g'$};
\draw[] (0,0.1) -> (1,0.8);
\draw[] (0,-0.1) -> (1.2,-.6);

\end{tikzpicture}
}\caption{}\label{soleado}
\end{figure}

\noindent the closed subspaces of the $\top$ topology are $\vacio$, $\{v\}$, $\{v,u\}$, $\{v,w\}$ and $E^0$. Meanwhile in the graph $F$ the closed ones are $\vacio$, $\{v'\}$, $\{v',u'\}$, $\{v',w'\}$ and $F^0$. We can define a homeomorphism $\tau: E^0 \rightarrow F^0$ by $\tau(a)=a'$ for every $a \in E^0$. Observe that $\overline{\{u,w\}}=\{u,w\} \neq E^0$ but $\overline{\{\tau(u),\tau(w)\}}=\overline{\{u',w'\}}=F^0$. This shows that hereditary and saturated subsets are not preserved under homeomorphisms.

\end{example}

In general in an arbitrary graph $E$ it satisfies that 
$$\big(\bigcup_{v\in c(w)} \{w\}\big)\setminus\{v\}=r({\hat s}^{-1}(v))\setminus\{v\}.$$
However if $E$ is an acyclic graph, for any vertex $v$ (not necessarily regular) we have $$\big(\bigcup_{v\in c(w)} \{w\}\big)\setminus\{v\}=r({\hat s}^{-1}(v)).$$
This implies that an hereditary $H$ is saturated if and only if 
$\big(\bigcup_{v\in c(w)} \{w\}\big)\setminus\{v\}\subset H$ implies $v\in H$.
Thus, for acyclic graphs, hereditary and saturated subsets are described in topological terms. So if $E^0$ and $F^0$ are homeomorphic as topological spaces and acyclic, the sets $\H_E$ and $\H_F$ are bijective: more precisely if $f\colon E^0\to F^0$ is a homeomorphism, then the map $f^*\colon\H_E\to \H_F$ such that 
$f^*(H)=f(H)$ is bijective. 

\begin{definition}\rm
Let $E$ be an arbitrary graph.
A vertex $v$ of $E^0$ is called an \emph{initial vertex} if 
\begin{equation}\label{eqiv}
    s(r^{-1}(v)) \subseteq\{v\}.
\end{equation}
If $v$ is initial, then any edge in $r^{-1}(v)$ is a loop. 
Let $n\in\N$, an \emph{inital $n$-looped vertex} $v$ is an initial vertex such that $\vert r^{-1}(v)\vert=n$. 
For instance, any source is an $0$-looped initial vertex.
A  vertex $v$ of $E^0$ is called an \emph{terminal vertex} if 
\begin{equation}\label{eqtv}
r(s^{-1}(v))\subseteq \{v\}.
\end{equation}
If $v$ is terminal, then any edge in $s^{-1}(v)$ is a loop.
A \emph{terminal} $n$-looped vertex $v$ is a terminal vertex such that  $\vert s^{-1}(v)\vert=n$. 
For example, any sink is a terminal $0$-looped vertex.\medskip

For instance, in graph $E$ below, the vertex $v$ is a initial $3$-looped vertex and in the graph $\op{E}$, the vertex $v$ is a terminal $3$-looped vertex.
\[
\xygraph{!{(0.25,0.5)}*+{E\colon}
!{<0cm,0cm>;<1cm,0cm>:<0cm,1cm>::}
!{(2,0)}*+{\cdots}="d"
!{(3,0)}*+{\bullet}="g"
!{(3,0.25)}*+{\hbox{\tiny $v$}}="v"
"g":"d" 
"g" :@(ul,ur) "g"
"g" :@(dl,dr) "g"
"g" :@(ur,dr) "g"
}\hskip2cm
\xygraph{!{(0.25,0.5)}*+{E^{\rm op}\colon}
!{<0cm,0cm>;<1cm,0cm>:<0cm,1cm>::}
!{(2,0)}*+{\cdots}="d"
!{(3,0)}*+{\bullet}="g"
!{(3,0.25)}*+{\hbox{\tiny $v$}}="v"
"d":"g" 
"g" :@(ul,ur) "g"
"g" :@(dl,dr) "g"
"g" :@(ur,dr) "g"
}\]

\end{definition}

\begin{remark}\rm
The formulas \eqref{eqiv} and \eqref{eqtv} are equivalent to the corresponding formulas in which $r,s$ are the natural extensions $r,s: \path{(E)} \rightarrow E^0$.
\end{remark}

\begin{proposition}
Let $v\in E^0$, then:
\begin{enumerate}
\item The following are equivalent:
\begin{enumerate}
    \item\label{iv} $v$ is initial.
    \item\label{fuente} $c(v)=\{v\}$.
    \item\label{alberca} $v$ is an initial $n$-looped vertex for some $n\in\N$.
\end{enumerate}
\item Analogously, these are equivalent: 
\begin{enumerate}
    \item\label{tv} $v$ is terminal.
    \item\label{s_1} $\bigcup_{v\in c(w)}\{w\}=\{v\}$.
    \item\label{s_2} $v$ is a terminal $n$-looped vertex for some $n \in \N$.
\end{enumerate}

\hskip -1.7cm In particular:
\item\label{initial} If $v$ is not an initial $n$-looped  vertex for any $n\ge 1$, then $v$ is a source if and only if $c(v)=\{v\}$.

\item\label{terminal} If $v$ is not an terminal $n$-looped  vertex for any $n\ge 1$, then $v$ is a sink if and only if $\left(\bigcup_{v\in c(w)}\{w\}\right)\setminus\{v\}=\vacio$.
\end{enumerate}
\end{proposition}
\begin{proof}

First for proving  \eqref{iv}$\Rightarrow$\eqref{fuente} suppose that $v$ is initial. Take $w \in c(v)$ and assume there is a path $\l$ such that $s(\l)=w$ and $r(\l)=v$. If $\l$ is a trivial path,  we have $w=v$ and $w=v \in c(v)$. Otherwise, $\l=f_1\ldots f_n$ with $r(f_n)=v$ and since $f_n \in r^{-1}(v)$ we get $s(f_n)=v$ so that $f_n$ is a loop based at $v$. In general, if $s(f_i)=v$ then applying the same argument we obtain $s(f_{i-1})=v$ and finally $f_i$ is a loop based at $v$ for every $i \in \{1,\ldots,n\}$. Therefore $w=v$. For \eqref{fuente}$\Rightarrow$\eqref{iv}, assume $c(v)=\{v\}$. Consider $w \in s(r^{-1}(v))$. If $w = v$ we are done. If $w \neq v$, then there exists an edge $f$ with $s(f)=w$ and $r(f)=v$. This implies $w \ge v$ and by hypothesis $w = v$ giving a contradiction. In conclusion $s(r^{-1}(v)) \subseteq \{v\}$, that is, $v$ is an initial vertex. Also observe that \eqref{iv}$\Leftrightarrow$\eqref{alberca} is straightforward.

For proving \eqref{tv}$\Rightarrow$\eqref{s_1} assume that $v$ is terminal and take $w\in E^0$ such that $v\ge w$, then either $v=w$ in which case we are done or there is a nontrivial path $\l$ from $v$ to $w$. Then the first arrow $f$ of $\l$ is in
$s^{-1}(v)$ hence $r(f)\in r(s^{-1}(v))\subseteq\{v\}$ and we have $r(f)=v$. Thus the first arrow of $\l$ is a loop. Applying this argument repeatedly, we get that $w=r(\l)=v$. 
For \eqref{s_1}$\Rightarrow$\eqref{tv} 
consider a vertex $u \in r(s^{-1}(v))$, and then there exists an edge $e$ such that $s(e)=v$ and $r(e)=u$. So $u \in \bigcup_{v\in c(w)}\{w\}=\{v\}$ giving $u=v$. Observe that \eqref{tv} is equivalent to $\eqref{s_2}$.
Finally, (\ref{initial}) and (\ref{terminal}) are  direct consequences of the previously proved items.
\end{proof}

\begin{remark}\label{pollo}{\rm We have the equality of the number of initial (respectively terminal) vertices in homeomorphic graphs $E$ and $F$, more precisely, homeomorphisms between graphs induce bijections between the sets of initial (resp. terminal) vertices in the corresponding graphs. And if $E$ and $F$ are graphs without initial (respectively terminal) $n$-looped vertices for $n\ge 1$, any homeomorphism between them induces a bijection between the sets of sources (respectively sinks) of $E$ and $F$. }
\end{remark}

Let $E$ be a graph, $u,v\in E^0$ and assume that there is an injective map 
$\theta\colon s^{-1}(u)\to s^{-1}(v)$  such that for each $r(f)=r(\theta(f))$ for any $f\in s^{-1}(u)$. Let $F=E(u\hookrightarrow v)$ be the \emph{shift graph} associated to $\theta$, that is, $F^0:=E^0$ and $F^1:=\{g\}\sqcup (E^1\setminus\hbox{Im}(\theta))$ with $g \notin E^1$, where $s_F(g)=v$, $r_F(g)=u$ and for any other arrow $f\in F^1$ we have $s_F(f)=s_E(f)$ and  $r_F(f)=r_E(f)$. For more information about the shift graph see \cite[Definition 2.1]{AALP}.
Define then the map $\varphi\colon E^0\to F^0$ given by $\varphi(v)=v$ for any
$v\in E^0$. 

\begin{theorem}\label{shiftcont} In the previous conditions $\varphi$ is continuous for the $\top$ topologies of $E^0$ and $F^0$.
\end{theorem}
\begin{proof} First we claim that the two following statements are equivalent:
\begin{itemize}
    \item[(a)] Every closed subset of $F^0$ is a closed subset of $E^0$.
    \item[(b)] For all $u \in E^0$ and for all $S$ closed subset of $F^0$, if $u \ge_{E} S$ then $u \ge_{F} S$.
\end{itemize}
Indeed, for proving $(a) \Rightarrow (b)$ let $u \in E^0$ and $S$ a closed subset of $F^0$. Since also $S$ is a closed subset of $E^0$ by hypothesis, we have that if $u \ge_{E} S$ then $u \in S$. And conversely for $(b) \Rightarrow (a)$, consider $S$ a closed subset of $F^0$. Now we know that  if $u \ge_{E} S$ then $u \ge_{F} S$. We check that $S$ is closed of $E^0$: if $v \ge_{E} S$ then $v \ge_{F} S$, implying that $v \in S$. Since  $v \ge_{E} S \Rightarrow v \in S$, we have that $S$ is closed of $E^0$.

In the next step, we prove the statement given in the theorem. We have to check that the set of closed subsets of $F^0$ is contained in the set of closed subsets of $E^0$. Let $S$ be a closed subset in $F^0$. To prove that $S$ is closed in the topology of $E^0$ it suffices to check that
$$\forall u\in F^0, ((\ u\ge_E S) \Rightarrow (u\ge_F S )).$$
We will prove something slightly stronger: that  if $u\ge_E v$, then $u\ge_F v$. Indeed, the unique arrows of $E$ that has been eliminated in $F$ are those in $\hbox{Im}(\theta)$. Assume 
$s^{-1}(u)=\{f_1,\ldots f_k\}$ and 
$s^{-1}(v)=\{h_1,\ldots h_k,\ldots h_n\}$ where $k\le n$ and $\hbox{Im}(\theta)=\{h_1,\ldots,h_k\}$. Also $r(f_i)=r(h_i)$ for $i=1,\ldots,k$.
Then 
$F^1=\{g\}\sqcup (E^1\setminus\{h_1,\ldots,h_k\})$ where $s(g)=v$ and $r(g)=u$ . However the elimination of the edges $h_1,\ldots h_k$ do not eliminate connections since $v$ connects in $F$ with $r(f_i)$ through the path $gf_i$ (for $i=1,\ldots,k$). 
\end{proof}

\begin{example}\rm Consider the graphs $E$ and $F$ given below. In $E$ the closed subsets of the $\top$ topology are $\vacio$, $\{u_1\}$, $\{u_1,u_2\}$ and $E^0$. On the other hand, in $F$ we have $\vacio$, $\{v_1\}$ and $F^0$.

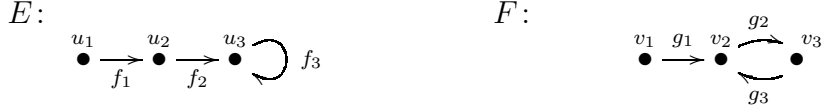
\[
\xygraph{!{(0.25,0.5)}*+{E\colon}
!{<0cm,0cm>;<1cm,0cm>:<0cm,1cm>::}
!{(1,0)}*+{\bullet}="c"
!{(2,0)}*+{\bullet}="d"
!{(3,0)}*+{\bullet}="g"
!{(3,0.25)}*+{\hbox{\tiny $u_3$}}="v"
!{(2,0.25)}*+{\hbox{\tiny $u_2$}}
!{(1,0.25)}*+{\hbox{\tiny $u_1$}}
!{(1.5,-0.25)}*+{\hbox{\tiny $f_1$}}
!{(2.5,-0.25)}*+{\hbox{\tiny $f_2$}}
!{(4,0)}*+{\hbox{\tiny $f_3$}}
"c":"d"
"d":"g" 
"g" :@(ur,dr) "g"
}\hskip2cm
\xygraph{!{(0.25,0.5)}*+{F\colon}
!{<0cm,0cm>;<1cm,0cm>:<0cm,1cm>::}
!{(2,0)}*+{\bullet}="d"
!{(3,0)}*+{\bullet}="g"
!{(4,0)}*+{\bullet}="h"
!{(2,0.25)}*+{\hbox{\tiny $v_1$}}
!{(2.5,0.25)}*+{\hbox{\tiny $g_1$}}
!{(3,0.25)}*+{\hbox{\tiny $v_2$}}="v"
!{(3.5,0.5)}*+{\hbox{\tiny $g_2$}}
!{(4.2,0.25)}*+{\hbox{\tiny $v_3$}}
!{(3.5,-0.5)}*+{\hbox{\tiny $g_3$}}
"d":"g" 
"g" :@(ur,dr) "h"
"h" :@(ur,dr) "g"
}\]

\medskip

The map $\varphi\colon E^0\to F^0$ such that $u_i\mapsto v_i$ for $i=1,2,3$ is continuous but not a homeomorphism since the image of the closed subset $\{u_1,u_2\}$ of $E^0$ is not a closed subset of $F^0$. However the canonical extension $\varphi\colon E\to F$ such that $\varphi(f_i)=g_i$ and $\varphi(f_i^*)=g_i^*$ for $i=1,2,3$ induces an isomorphism of Leavitt path algebras from $L_K(E)$ to $L_K(F)$
(in fact a shift move).
\end{example}

Before finishing this section, we have to introduce some formalities about categories in the next subsection.

\subsection{Graph Categories}\label{cati}

\def\grph{\mathcal{Grph}}
\def\set{\mathcal{Set}}
\def\f{\mathfrak{F}}
This subsection arises from the need to define \lq\lq operators\rq\rq\ which can be applied to any graph and produce certain sets. So, for instance, the assignation $E\mapsto\h_E$ mapping any graph with its set of hereditary and saturated sets is an example. Also one can map any graph $E$ to its set of line-points: $E\mapsto \Pl(E)$. So we can think of $\Pl$ as an operator acting on the class of all graphs. To way to formalize these examples is by using functors, so: category theory. 

As usual, for a category $\mathcal C$, the notation $X\in\mathcal C$ means that $X$ is an object of the category. 
When defining functors among categories, usually we will define only the object function when the morphism one is clear.
Define by $\grph$ the category whose objects are the directed graphs and for $E,F\in\grph$, we define $\hom_{\tiny\grph}(E,F)$ as the set of all isomorphisms (if any) $E\to F$. 
Denote by $\set$ the category of sets.
We will have the occasion of dealing with functors $\grph\to\set$. For instance $\h\colon\grph\to\set$ such that $E\mapsto\h_E$ the set of all hereditary and saturated subsets of $E^0$. We will also use the functors 
$\mathcal{E}^0,\mathcal{E}^1\colon \grph\to\set$ such that $\mathcal{E}^0(E)=E^0$ and $\mathcal{E}^1(E)=E^1$. We also define 
functors $\Pl,\Pc,\Pec,\Pb\colon\grph\to \set$ by writing $\Pl(E):=\{\hbox{line-points of}\ E\}$,
$\Pc(E):=\{\hbox{vertices in cycles without exits of}\ E\}$,
$\Pec(E):=\{\hbox{vertices in extreme cycles of}\ E\}$, and 
$$\Pb(E):=\{\hbox{vertices of $E$ whose tree contains infinite bifurcations}\}.$$
\begin{definition}\rm
A functor $H\colon\grph\to\set$ is said to be a \emph{point functor} if $H(E)\subset E^0$ (in other words, if it is a subfunctor of $\mathcal{E}^0$ defined above).
\end{definition}
Denote by $\f(\grph,\set)$ the category whose objects are the functors $\grph\to\set$ and for $H,J\in\f(\grph,\set)$ a morphism from $H$ to $J$ in 
$\f(\grph,\set)$ is a natural transformation $\tau\colon H\to J$. Thus $\tau=(\tau_E)_{E\in\grph}$ where 
$\tau_E\colon H(E)\to J(E)$ for any graph $E$, and the squares

\begin{center}
\begin{tikzcd}
H(E) \arrow[r, "\tau_E"] \arrow[d, "H(\a)"']
& J(E) \arrow[d, "J(\a)"] \\
H(E') \arrow[r,  "\tau_{E'}"']
& J(E')
\end{tikzcd}
\end{center}

\noindent commute when $\a\in \hom_{\grph}(E,E')$. We also define the category $\f^\sharp(\grph,\set)$ as the full subcategory of $\f(\grph,\set)$ whose objects are the point functors. 
\begin{definition}\rm
A point functor $H\colon\grph\to\set$ is said to be \emph{hereditary} in case $H(E)$ is a hereditary subset of $E^0$ for any graph $E$. Similarly can we define \emph{hereditary saturated} point functors $\grph\to\set$.
\end{definition}

Given a point functor  $H\colon\grph\to\set$ we define its hereditary closure denoted 
$\overline{H}^h$ as the new point functor $\overline{H}^h\colon\grph\to\set$ given by  $\overline{H}^h(E):=\hbox{Hereditary closure of } H(E)$ in $E^0$.
Similarly can we define the hereditary and saturated closure of a point functor $H$ (which we will denote by $\overline{H}$). We have the usual relations $H\subset \overline{H}^h\subset \overline{H}$ in the sense of subfunctors.  We can think of point functors as if they were ordinary subsets of vertices in a graph. So given two point functors $H_1,H_2$ we can construct in a obvious way the boolean operations $H_1\cup H_2$, $H_1\cap H_2$, $H_1 \setminus H_2$.
In particular, if $H$ is a point functor, we can construct a new point functor 
${\mathcal E}^0 \setminus H\colon\grph\to\set$ given by 
$({\mathcal E}^0 \setminus H)(E):=E^0\setminus H(E)$.
\begin{definition}{\rm
A point functor $H\colon\grph\to\set$ is said to be {\em closed (respectively open)} if $H(E)$ is closed in the $\top$ topology of $E$ (respectively open). Also given $H$ we can define new functors 
$c(H)$, $\accentset{\circ}{H}$, $\ext(H)$, $\partial H\colon\grph\to\set$, 
given by $$c(H)(E):=c(H(E)),\ \accentset{\circ}{H}(E):=\accentset{\circ}{\aoverbrace[L1R]{H(E)}}, \ \ext(H)(E):=\ext(H(E)),\  (\partial H)(E):=\partial(H(E)).$$
These new functors may be referred to by their usual names: closure of $H$ denoted $c(H)$,
interior of $H$ denoted $\accentset{\circ}{H}$, exterior of $H$ denoted $\ext(H)$, and boundary of $H$ denoted $\partial H$.}
\end{definition}

Observe that $\ext(H)$ can be described in terms of the connection of vertex by 
\begin{equation}\label{babero}
\ext(H)(E)=\{v\in E^0 \ \vert \ v\not\ge H(E)\}.
\end{equation}
We will see later on that when the functor $H$ is hereditary and saturated then so is $\ext(H)$ (see Proposition \ref{here}). We also have the following:

\begin{definition}\label{hervor}\rm
Let $f\colon L_K(E)\to L_K(F)$ be a ring isomorphism. Given two point functors  $H_i\colon\grph\to\set$ ($i=1,2$),   we will say that $H_1$ is  \emph{$f$-related} to $H_2$ if and only if $f(I(H_1(E)))=I(H_2(F))$. We will say that a point functor $H$ is \emph{$f$-invariant} if and only if $H$ is $f$-related to itself, that is, 
$f(I(H(E)))=I(H(F))$. Finally, a point functor $H$ is said to be {\em invariant under isomorphism} if and only if is $f$-invariant for any isomorphism $f$.
\end{definition}

Note that for $A,B\in\H_E$ one has $I(A\cup B)=I(A)+I(B)$ (the idea of the proof is 
in \cite[Proposition 1.6]{CMMSS}). Also $I(A\cap B)=I(A)\cap I(B)$, the inclusion  $I(A\cap B)\subset I(A)\cap I(B)$ is straightforward and for the other $I(A)\cap I(B)=I(C)$ for a suitable $C\in\H_E$.
Then $C\subset A\cap B$ hence $I(A)\cap I(B)=I(C)\subset I(A\cap B)$.

\begin{proposition}\label{union}
Let $H_i\colon\grph\to\set$ ($i=1,2$) be hereditary and saturated point functors and let $f\colon L_K(E)\to L_K(F)$ be an isomorphism. If $H_1$ and $H_2$ are $f$-invariant, then $H_1\cup H_2$ and $H_1\cap H_2$ are $f$-invariant point functors. 
\end{proposition}

\begin{proof}
For the union, first observe that $I(H_1(E) \cup H_2(E))= I(H_1(E))+I(H_2(E))$. So applying $f$ to both sides of the last equality and our hypothesis we have: $f(I(H_1(E) \cup H_2(E)))= f(I(H_1(E))+I(H_2(E)))= f(I(H_1(E)))+f(I(H_2(E))) = I(H_1(F)) + I(H_1(F)) = I(H_1(F) \cup H_2(F))$ as desired. Now for the intersection, take into account that $I(H_1(E) \cap H_2(E))= I(H_1(E)) \cap I(H_2(E))$ and repeating the same argument then $f(I(H_1(E) \cap H_2(E)))=  I(H_1(F) \cap H_2(F))$.
\end{proof}

It has been proved that certain ideals associated to remarkable hereditary and saturated subsets of vertices are invariant under isomorphism of Leavitt path algebras.
Among these ideals we have: 
\begin{itemize}
\item[(1)] the ideal generated by $\Pl$ the set of line points since $I(\Pl)$ is the socle of the Leavitt path algebra (\cite[Theorem 4.2]{AMMS1}); 
\item[(2)] the ideal generated by $\Pc$ the set of vertices in cycles with no exits (\cite[Theorem 6.11]{ABS}); 
\item[(3)] the ideal generated by $\Pec$ the set of vertices in extreme cycles (\cite[Corollary 5.10]{CGKS});
\item[(4)] the ideal generated by $P_{ppi}$ the set of vertices which generates the largest purely infinite ideal of the Leavitt path algebra (\cite[Corollary 4.14]{CGKS}) and
\item[(5)] the ideal generated by $P_{ex}$ the set of vertices which generates the largest exchange ideal of a Leavitt path algebra (\cite[Corollary 6.3]{CGKS}). 
\end{itemize}

The point functor $\Pb$ is not invariant: the ideal generated by  $\Pb(E)$ (vertices whose tree contains infinite bifurcations) is not preserved under isomorphism in general, as the following example shows.

\begin{example}\label{ohcac}
\rm Consider the graphs $E$ and $F$ in Figure \ref{lluvioso}:
\begin{figure}[ht]
    \centering
   \hbox{\hskip 3cm
\begin{tikzpicture}[->]
\draw[fill=black] (-0.1,0) circle (1pt);
\draw[fill=black] (1.1,0.85) circle (1pt);
\draw[fill=black] (1.1,-.5) circle (1pt);
\draw[fill=black] (2,-.8) circle (1pt);
\draw[fill=black] (2,-.1) circle (1pt);
\draw[fill=black] (2.6,0.2) circle (0.5pt);
\draw[fill=black] (2.8,0.2) circle (0.5pt);
\draw[fill=black] (3,0.2) circle (0.5pt);
\draw[fill=black] (2.6,-0.2) circle (0.5pt);
\draw[fill=black] (2.8,-0.2) circle (0.5pt);
\draw[fill=black] (3,-0.2) circle (0.5pt);
\draw[fill=black] (2.6,-0.6) circle (0.5pt);
\draw[fill=black] (2.8,-0.6) circle (0.5pt);
\draw[fill=black] (3,-0.6) circle (0.5pt);
\draw[fill=black] (2.6,-1) circle (0.5pt);
\draw[fill=black] (2.8,-1) circle (0.5pt);
\draw[fill=black] (3,-1) circle (0.5pt);

\node at (-1,0){$E$:};
\node at (-0.4,0) {\tiny $v$};
\node at (1.1,-0.2) {\tiny $w_1$};
\node at (2,-1.1) {\tiny $w_2$};
\node at (2,0.2) {\tiny $w_3$};
\node at (1.3,1){\tiny $u$};
\node at (0.4,0.6) {\tiny $f$};
\node at (0.4,-0.5) {\tiny $g$};
\draw[] (0,0.1) -> (1,0.8);
\draw[] (0,-0.1) -> (1,-.5);
\draw[ ] (1.2,-0.4)--(1.9,-0.1);
\draw[ ] (1.2,-0.6)--(1.9,-0.8);
\draw[ ] (2.1,-0.8)--(2.5,-0.6);
\draw[ ] (2.1,-0.8)--(2.5,-1);
\draw[ ] (2.1,-0.1)--(2.5,-0.2);
\draw[ ] (2.1,-0.1)--(2.5,0.2) ;
\end{tikzpicture}
\hskip 1cm
\begin{tikzpicture}[->]
\draw[fill=black] (-0.1,0) circle (1pt);
\draw[fill=black] (-0.1,1) circle (1pt);
\draw[fill=black] (1.1,1) circle (1pt);
\draw[fill=black] (1.1,-.5) circle (1pt);
\draw[fill=black] (2,-.8) circle (1pt);
\draw[fill=black] (2,-.1) circle (1pt);
\draw[fill=black] (2.6,0.2) circle (0.5pt);
\draw[fill=black] (2.8,0.2) circle (0.5pt);
\draw[fill=black] (3,0.2) circle (0.5pt);
\draw[fill=black] (2.6,-0.2) circle (0.5pt);
\draw[fill=black] (2.8,-0.2) circle (0.5pt);
\draw[fill=black] (3,-0.2) circle (0.5pt);
\draw[fill=black] (2.6,-0.6) circle (0.5pt);
\draw[fill=black] (2.8,-0.6) circle (0.5pt);
\draw[fill=black] (3,-0.6) circle (0.5pt);
\draw[fill=black] (2.6,-1) circle (0.5pt);
\draw[fill=black] (2.8,-1) circle (0.5pt);
\draw[fill=black] (3,-1) circle (0.5pt);

\node at (-1.3,0){$F$:};
\node at (-0.4,1) {\tiny $v_1$};
\node at (-0.4,0) {\tiny $v_2$};
\node at (1.1,-0.1) {\tiny $w_1'$};
\node at (2,-1.1) {\tiny $w_2'$};
\node at (2,0.2) {\tiny $w_3'$};
\node at (1.3,1.1){\tiny $u'$};
\node at (0.4,1.2) {\tiny $f'$};
\node at (0.4,-0.5) {\tiny $g'$};
\draw[] (0,1) -> (1,1);
\draw[] (0,-0.05) -> (1,-.5);
\draw[ ] (1.2,-0.4)--(1.9,-0.1);
\draw[ ] (1.2,-0.6)--(1.9,-0.8);
\draw[ ] (2.1,-0.8)--(2.5,-0.6);
\draw[ ] (2.1,-0.8)--(2.5,-1);
\draw[ ] (2.1,-0.1)--(2.5,-0.2);
\draw[ ] (2.1,-0.1)--(2.5,0.2) ;
\end{tikzpicture}
} \caption{}
    \label{lluvioso}
\end{figure}
\bigskip

We assume that  $T(w_1)=\{w_i\}_{i\ge 1}$ and each $w_i$ is a bifurcation with two edges for the graph $E$ and similarly for the graph $F$. Thus $\Pb(E)=\{v\}\cup\{w_i\}_{i\ge 1}$ and $I(\Pb(E))=L_K(E)$. On the other hand $\Pb(F)=\{v_2\}\cup\{w_i'\}_{i\ge 1}$ and the ideal $I(\Pb(F))$ is not $L_K(F)$. In fact $L_K(F)/I(\Pb(F))\cong M_2(K)$.
\end{example}

However we have: 
\begin{proposition} Let $E$ and $F$ be the graphs considered in the above Example \ref{ohcac}. There is a graded $*$-isomorphism of $K$-algebras $\theta\colon L_K(E)\to L_K(F)$ such that $\theta(v)=v_1+v_2$ and the image under $\theta$ of the other vertices and edges are the homonymous vertices and edges of $F$ (and the same applies to ghost edges).
\end{proposition}
\begin{proof}
The existence of the isomorphism is based upon the \lq\lq out-split\rq\rq\ move (see \cite{AALP}). However we describe the construction of the isomorphism. 
We define first the linear map $\psi\colon K \hat E\to L_K(F)$ such that 
$\psi(v)=v_1+v_2$ and the image under $\psi$ of the other vertices and edges (real or ghost) are the homonymous vertices and edges of $F$ (as elements of $L_K(F)$). Also the image of a nontrivial path $x_1\cdots x_n$ in $K\hat E$ is defined to be $x_1'\cdots x_n'$.
Then we prove
that for $t\in\hbox{reg}(E)$, each difference $t-\sum hh^*$ (sum extended to edges $h$ with $s(h)=t$) maps to $0$ under $\psi$. This induces by passing to the quotient a homomorphism of $K$-algebras $\theta$ from $L_K(E)$ to $L_K(F)$. This homomorphism is an epimorphism since all the generators of $L_K(F)$ are in the image of $\theta$: for instance $v_1=\theta(ff^*)$ and $v_2=\theta(gg^*)$.
To see that $\theta$ is a monomorphism, observe that $E$ satisfies Condition (L) and we apply the Cuntz-Krieger Uniqueness theorem (see \cite[Theorem 2.2.16]{AAS}). The given isomorphism is actually a $*$-isomorphism by construction and it is also a graded isomorphism.
\end{proof}

\begin{remark}\rm  Note that according to \cite[Proposition 2.6]{CGKS}, $I(\Pec \cup \Pb)$ is invariant under any ring isomorphism. But this is not true in general because in that proof it is strongly used that the ideal $I(\Pec \cup \Pb)$ does not contain any primitive idempotents. For instance, in the graph below consisting of one \lq\lq fiber\rq\rq\ 

\hskip 7cm
\xygraph{
!{<0cm,0cm>;<1.5cm,0cm>:<0cm,1.2cm>::}
!{(0,0) }*+{\bullet_{u}}="a"
!{(1.5,0) }*+{\bullet_{v}}="b"
!{(0,1.5)}*+{}
"a":^{f_n}_{(\infty)}"b" 
}      
\vskip .5cm

\noindent there is a sink $v$ which is a primitive idempotent and it belongs to the ideal $I(\Pec \cup \Pb)$. Such primitive idempotents (belonging to $I(\Pb))$ may also be present in row-finite graphs (see the graph $E$ in Example \ref{ohcac}).
\end{remark}

In this work we deal with suitable sets of vertices which define invariant ideals. We will prove in a forthcoming section that for any isomorphism $f\colon L_K(E)\to L_K(F)$,
and for any $f$-invariant hereditary and saturated functor $H$, the exterior 
$\ext(H)$ is again $f$-invariant. However the other functors (interior, closure, etc.) are not necessarily $f$-invariant.

\begin{example}{\rm   The following example shows that $\partial H = c(H) \cap c(H^c)$ is not invariant via isomorphism. Consider the graphs given in \ref{ohcac}. In the graph $E$ take $H(E)=\Pl(E)=\{u\}$. We have that $c(H(E))=\{u,v\}$ and $c(H(E)^c)=\{v\}\cup \{ w_i\}_{i \geq 1}$ and so $\partial H(E)=\{v\}$. On the other hand, $H(F)=\Pl(F)=\{v_1, u'\}$ and $c(H(F))=\{v_1,u'\}$ and $c(H(F)^c)=\{v_2\} \cup \{w_i'\}_{i \geq 1}$. Finally $\partial H(F)=\vacio$.
}
\end{example}

\section{Annihilators}\label{anni}
For an arbitrary algebra $A$ (not necessarily associative) and an ideal $I\triangleleft A$, we can consider the {\it annihilator} $\ann(I):=\{a\in A\colon aI=Ia=0\}$. This is an ideal of $A$ and we have $I\subset\ann(\ann(I))$. Also it is easy to see that $\ann(\ann(\ann(I)))=\ann(I)$ for any ideal of $A$.
Let us denote $\Tilde I:=\ann(\ann(I))\supset I$ for any ideal $I$ of $A$.

Now, we consider the definition of regular ideal in the sense of \cite{Hamana}. These ideals are recently studied in \cite{DanielDanilo} in the context of Leavitt path algebras.

\begin{definition}\rm Let $A$ be a $K$-algebra,
an ideal $I\triangleleft A$ satisfying 
$\Tilde{I}=I$ is called {\it regular ideal}. 
\end{definition}
It is easy to see that the ideals of the form $I=\ann(J)$ (for another ideal $J$) are regular.

After writing Proposition \ref{here} below and Corollary \ref{qite}, we learn about the work \cite{DanielDanilo} whose Proposition 3.5 contains a similar result.

\begin{proposition}\label{here} Let $H \in \mathcal{H}_E$ and define $H'=\{v \in E^0 \; | \; v \not\geq H \}$. Then:
\begin{enumerate}
\item $H'$ is a hereditary and saturated subset of $E^0$, that is, $H' \in \mathcal{H}_E$.
\item ${\rm Ann}(I(H))= I(H')$.
\end{enumerate}
\end{proposition}

\begin{proof}
For the first part, take $v\in H'$ and assume $v\ge v'\in E^0$. If $v'\ge H$ then $v\ge H$ a contradiction. So $H'$ is hereditary. To prove that it is saturated consider a vertex $v$ such that $r(s^{-1}(v))\subset H'$. Let $\lambda\in\hbox{\rm path}(E)$ with $s(\lambda)=v$ and $r(\lambda)\in H$. Then 
writing $\lambda=f_1\cdots f_n$ we have $r(f_1)\in H'$ hence $r(f_1)\not\ge H$. But on the other hand $r(f_1)\ge r(\lambda)\in H$, a contradiction. This proves that any path whose source is $v$ has target out of $H$. Whence $v\not\ge H$ so that $v\in H'$. Let us prove now the second item. Take $u\in H'$ and let us check that $u I(H)=0$. If $z\in I(H)$ we can write $z=\sum_i k_i\a_i\b_i^*$ with $k_i \in K$ and $\a_i,\b_i$ paths whose range is in $H$. In case $uz\ne 0$ there must be some $i$ such that $u\a_i\b_i^*\ne 0$. Then  
$u\ge r(\a_i)\in H$, a contradiction. Hence $H'I(H)=0$ and applying the canonical involution $I(H)H'=0$. Consequently $H'\subset\ann(I(H))$ implying $I(H')\subset\ann(I(H))$.
Conversely, let $z\in \ann(I(H))$ be an homogeneous element. We will prove first that for any vertex $u$ such that $u\in \ideal(z)$ one has $u\in H'$.
Indeed: $\ideal(z)\subset \ann(I(H))$ whence $uI(H)=I(H)u=0$. If $u\ge H$ there is a path $\lambda$ with 
$s(\lambda)=u$ and $r(\lambda)\in H$. But then $\lambda=u\lambda\in u I(H)=0$ a contradiction. Thus $u\in H'$. So far we have $H_1:=E^0\cap \ideal(z)\subset H'$. So $I(H_1)\subset I(H')$.
Moreover, $\ideal(z)=I(H_1)$ so we deduce that
$\ideal(z)\subset I(H')$. But this is true for any homogeneous element $z\in \ann(I(H))$ hence for any element of $\ann(I(H))$. So $\ann(I(H))\subset I(H')$. \end{proof}

\begin{corollary}\label{qite} Let $H \in \h_E$ and $H'=\{v \in E^0 \; | \; v \not\geq H \}$. Define $H''=\{v \in E^0 \; | \; v \not\geq H' \}$.  Then:
\begin{enumerate}
\item $H'' \in \mathcal{H}_E$ and ${\rm Ann}({\rm Ann}(I(H)))=I(H'')$.
\item $H'' \subseteq \{v \in E^0 \; \vert \; v \geq H \}$.
\item $H \subset H''$.
\item $H'' \subset H$ if and only if for any $v\in E^0$ one has $v\not\ge w$ for every $w\not\ge H$ implies $v \in H$.

\end{enumerate}
\end{corollary}
\begin{proof}
The first item is straightforward from  Proposition \ref{here}. For the second if $v\in H''$, then $v \not \ge H'$ implies $v \notin H'$, so $v \ge H$. 
For proving (3),  if $v\in H$ and $v \not \in H''$ then $v \ge H'$ which is a contradiction. 
To prove (4) suppose first that $H'' \subset H$. Let $v$ be such that $v\not\ge w$ for every $w\not\ge H$. Then $v \not \ge H'$ which implies $v \in H''$. Because of the assumption $H'' \subset H$, we have $v \in H$. For the converse, consider $v \in H''$. So $v \not \ge H'$, that is, $v\not\ge w$ for every $w\not\ge H$. Then $v \in H$. \qedhere
\end{proof}

\begin{remark}\rm 
Observe that by  the Proposition \ref{here}  it is easy to check that an ideal $I(H)$ is regular if and only if $H''\subset H$ if and only if for any $v\in E^0$ one has 
$v\not\ge w$ for every $w\not\ge H$ implies $v \in H$.
\end{remark}

If $f\colon A\to B$ is a ring isomorphism we know that for any ideal
$I\triangleleft A$ one has $f(\ann(I))=\ann(f(I))$. Therefore for any $I\triangleleft A$ one has 
$f(\Tilde I)=\widetilde{f(I)}$, implying that $f$ transforms regular ideals into regular ideals. 

\begin{proposition}
If $f\colon L_K(E)\cong L_K(F)$ is a ring isomorphism and $H_1\in\h_E$, $H_2\in\h_F$, with 
$f(I(H_1))=I(H_2)$. Then (following the notation in Corollary \ref{qite}) we
have $f(I(H_1'))=I(H_2')$.
\end{proposition}
\begin{proof}
We know that $I(H_i')=\ann(I(H_i))$ for $i=1,2$. So $$f(I(H_1'))=f(\ann(I(H_1))=\ann(f(I(H_1)))=\ann(I(H_2))=I(H_2').$$
\end{proof}

The results in this section are re-stated in terms of point
functors in the proposition below. We highlight that our main interest is in invariant point functors so item \eqref{leo} is
an essential result for our purposes.

\begin{proposition}\label{dormit}
Let $H\colon\grph\to\set$ be a hereditary and saturated point functor. Then: 
\begin{enumerate}
\item $\ext(H)$ is a hereditary and saturated point functor. 
\item $\ann(I(H(E))=I(\ext(H)(E))$.
\item $\ext(\ext(H))$ is a hereditary and saturated point functor and $\ann(\ann(I(H(E)))= I(\ext(\ext(H(E)))$.
\item $\ext(\ext(H(E)) \subseteq \{v \in E^0 \; \vert \; v \geq H(E) \}$.
\item $H(E) \subset \ext(\ext(H(E))$.
\item $\ext(\ext(H(E)) \subset H(E)$ if and only if for any $v\in E^0$ one has $v\not\ge w$ for every $w\not\ge H(E)$ implies $v \in H(E)$.
\item\label{leo} If $H_i\colon\grph\to\set$ ($i=1,2$) are point functors and $H_1$ is $f$-related to $H_2$ then $\ext{(H_1)}$ is $f$-related to $\ext{(H_2)}$. In particular if a point functor $H$ is $f$-invariant, then also $\ext{(H)}$ is $f$-invariant.
\end{enumerate}
\end{proposition}

Since the functors $\Pl$, $\Pc$ and $\Pec$ are invariant, then the functors $\ext(\Pl)$, $\ext(\Pc)$ and $\ext(\Pec)$ are also invariant. Furthermore, since 
$\Plce:=\Pl\cup\Pc\cup\Pec$ is invariant by Proposition \ref{union}, we have that $\ext(\Plce)$ is also invariant. Observe that $\ext(\Plce)$ is a subfunctor of $\Pb$. Concretely
$$\ext(\Plce)(E)=\{v\in E^0 \; \vert  \; v\not\ge \Plce\}.$$
\begin{definition}{\rm
For a graph $E$ we define the set of {\it vertices with pure infinite bifurcations}  $\Pbp(E):=\{v\in E^0 \; \vert \; v\not\ge\Plce\}$ and the point functor 
$\Pbp\colon\grph\to\set$ such that $E\mapsto \Pbp(E)$.}
\end{definition}
\begin{theorem}\label{enjundia}
Let $E$ be a graph, the functor $\Pbp\colon\grph{}\to \set$ is invariant.
\end{theorem}
\begin{proof}
By Proposition \ref{dormit}\eqref{leo}, it suffices to realize that $\Pbp=\ext(\Plce)$ and $\Plce$ is invariant.
\end{proof}

For instance, in the graph $E$
 of Example \ref{ohcac}, we have
 $\Pl=\{u\}$, $\Pc=\vacio$, $\Pec= \vacio$ and $\Pbp=\{w_i\}_{i\ge 1}$
 hence the ideals $I(\{u\})$ and 
 $I(\{w_i\})$ are invariant under ring isomorphisms. 
 
 \section{Socle chain in Leavitt path algebras}\label{soclechain}

 If $R$ is a ring and $M$ an $R$-module, one can define the series of socles of $M$ in the usual way: it is an ascending chain of $R$-submodules  $\{\soc^{(i)}(M)\}_{i\ge 1}$  where $\soc^{(1)}(M):=\soc(M)$ and $\soc^{(n+1)}(M)/\soc^{(n)}(M)=\soc(M/\soc^{(n)}(M))$. In particular this can be applied to an algebra $A$ so that the socle series defined in \cite{ARS} is a sequence of ideals 
 $$\soc(A)\subset\cdots\subset\soc\nolimits^{(n)}(A)\subset\soc\nolimits^{(n+1)}(A)\subset\cdots \quad (n\in\N)$$
such that 
$$\frac{\soc^{(n+1)}(A)}{\soc^n (A)}=\soc\left(\frac{A}{\soc^n(A)}\right).$$
We will focus on $n\in\N$ to avoid dealing with infinite cardinals. One of our goals in this section is to check that the different ideals $\soc^{(n)}(A)$ associated to a Leavitt path algebra $A$ are invariant under isomorphism and to characterize them graphically. The other purpose is more ambitious: since the socle is the ideal generated by a point functor, namely $\Pl$, we
would like to prove that there is a ascending chain of point functors (starting at $\overline{\Pl}$) all of which are invariant. From this point, we want to extrapolate so that we can apply this circle of ideas to other point functors (for instance $\Pc$) and even further, for any invariant hereditary and saturated functor (next in Section \ref{rosa}).

Consider Example 2.7 of \cite{ARS}

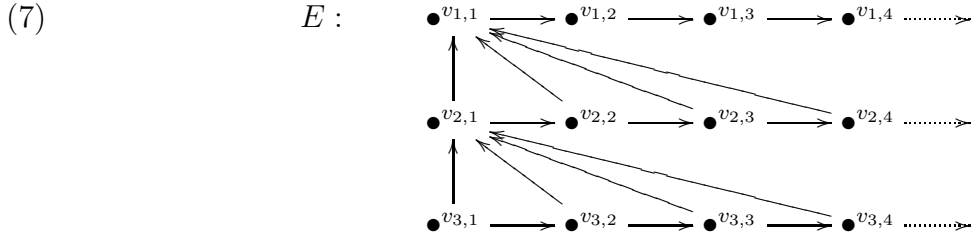
 \begin{equation} 
    \xymatrix{ E: & {\bullet}^{v_{1,1}} \ar[r]  & {\bullet}^{v_{1,2}} \ar[r]  & {\bullet}^{v_{1,3}} \ar[r]  &
               {\bullet}^{v_{1,4}} \ar@{.>}[r] &  \\
             & {\bullet}^{v_{2,1}} \ar[r] \ar[u] & {\bullet}^{v_{2,2}} \ar[r] \ar[ul] & {\bullet}^{v_{2,3}} \ar[r] \ar[ull] & {\bullet}^{v_{2,4}}\ar@{.>}[r] \ar[ulll] &  \\
             & {\bullet}^{v_{3,1}} \ar[r] \ar[u] & {\bullet}^{v_{3,2}} \ar[r] \ar[ul] & {\bullet}^{v_{3,3}} \ar[r] \ar[ull] & {\bullet}^{v_{3,4}}\ar@{.>}[r] \ar[ulll] &  \\
           }\label{ahora}
           \end{equation}
\medskip

Let $A$ be the Leavitt path algebra $L_K(E)$ where $E$ is the graph in  \eqref{ahora}.  We have by  $\soc(A)=I(P_l)$, with $P_l=\{v_{1,i}\}_{i\ge 1}$ \cite[Theorem 5.2]{AMMS2}. Then  $\frac{A}{\socs(A)}\cong L_K(F)$ where $F$ is the graph in \eqref{ahora2}:

\begin{equation} 
    \xymatrix{ F: 
             & {\bullet}^{v_{2,1}} \ar[r]  & {\bullet}^{v_{2,2}} \ar[r] & {\bullet}^{v_{2,3}} \ar[r]  & {\bullet}^{v_{2,4}}\ar@{.>}[r]  &  \\
             & {\bullet}^{v_{3,1}} \ar[r] \ar[u] & {\bullet}^{v_{3,2}} \ar[r] \ar[ul] & {\bullet}^{v_{3,3}} \ar[r] \ar[ull] & {\bullet}^{v_{3,4}}\ar@{.>}[r] \ar[ulll] &  \\
           }\label{ahora2}
           \end{equation}
\medskip
Thus $\soc\left(\frac{A}{\socs(A)}\right)\cong\soc(L_K(F))$ and since
$\Pl(F)=\{ v_{2,i}\}_{i\ge 1}$, so we have $\soc^{(2)}(A)= I(H)$ being $H=\{v_{1,i}\}_{i\ge 1}\cup \{v_{2,i}\}_{i\ge 1}$. Let $B=\frac{A}{\socs^{(2)}(A)}$, then $B\cong L_K(G)$, being $G$ the graph:

\begin{equation} 
    \xymatrix{ G: 
             & {\bullet}^{v_{3,1}} \ar[r]  & {\bullet}^{v_{3,2}} \ar[r]  & {\bullet}^{v_{3,3}} \ar[r]  & {\bullet}^{v_{3,4}}\ar@{.>}[r]  &  \\
           }\label{ahora3}
           \end{equation}
\medskip
Since 
$$\frac{\soc^{(3)}(A)}{\soc^{(2)}(A)}=\soc\left(\frac{A}{\soc^{(2)}(A)}\right)=\soc(B)=B=\frac{A}{\soc^{(2)}(A)},$$ we have $\soc^{(3)}(A)=A=I(E^0)$.
In this example the series of socles is $\soc(A)\subsetneq\soc^{(2)}(A)\subsetneq\soc^{(3)}(A)=A$, inducing a series of hereditary and saturated subsets 
$$\{v_{1,i}\}_{i\ge 1}\subsetneq 
\{v_{1,i}\}_{i\ge 1}\cup \{v_{2,i}\}_{i\ge 1}\subsetneq E^0
$$
and each of the hereditary saturated subsets in this series induces an ideal invariant under isomorphisms. This example illustrates the general phenomenon that we analyze in the following paragraph.
\bigskip

If $A=L_K(E)$ is a Leavitt path algebra then each ideal $\soc^{(n)}(A)$ is graded by \cite[Theorem 3.2]{ARS}. Applying \cite[Theorem 2.4.8]{AAS} we get $\soc^{(n)}(A)=I(\Pl^{(n)}(E))$ for a certain hereditary and saturated subset named $\Pl^{(n)}(E)$. Also, since $\soc^{(n)}(A)\subset\soc^{(n+1)}(A)$ we have $$\Pl\nolimits^{(n)}(E)\subset \Pl\nolimits^{(n+1)}(E),$$
for any graph $E$. Looking at $\Pl^{(n)}$ as point functors
$\Pl^{(n)}\colon\grph\to\set$, we have a sequence of hereditary and saturated functors 
$$\overline{\Pl}=\Pl\nolimits^{(1)}\subset\Pl\nolimits^{(2)}\subset\cdots$$
It is easy to see that if $f\colon A\to B$ is a ring isomorphism, then $f(\soc^{(n)}(A))=\soc^{(n)}(B)$. Summarizing we derive the following proposition.

\begin{proposition}\label{invariante}
The series of functors $\Pl^{(n)}$, ($n\ge 1$) are invariant under isomorphism in the sense of definition \ref{hervor}.
\end{proposition}

\begin{remark}\label{tita Antonia} {\rm In general, for a row-finite graph $E$ and $H(E) \in \mathcal{H}_E$, let $\theta$ be the isomorphism given in \cite[Corollary 2.4.13 (i)]{AAS}, that is, $\theta: L_K(E)/I(H(E)) \rightarrow  L_K(E/H(E))$. Remember, under this situation, we have $\theta^{-1}$ defined as follows: for $v \in (E/H(E))^0$ and $e \in (E/H(E))^1$, $\theta^{-1}(v) = v + I(H(E))$, $\theta^{-1}(e) = e + I(H(E))$ and $\theta^{-1}(e^{\ast}) = e^{\ast} + I(H(E))$. For short we will identify (without mentioning) an element in $L_K(E)/I(H(E))$ with its corresponding image through $\theta$ inside $L_K(E/H(E))$.} 
\end{remark}

Given that the hereditary saturated functors $P_l^{(n)}$ induce invariant ideals, we now consider the problem of describing in purely graph-theoretic terms, the sets $\Pl^{(n)}(E)$, ($n\ge 1$).
Consider the following diagram where $i_E\colon E^0\to L_K(E)$ is the canonical injection and $\pi$ the canonical projection $\pi\colon L_K(E)\to L_K(E)/I(\Pl^{(n)}(E))\cong L_K(E/\Pl^{(n)}(E))$ (up to identification). The elements of $L_K(E)/I(\Pl^{(n)}(E))$ will be denoted $x+I(\Pl^{(n)}(E))$ as usual.
Denote $F:=E/\Pl^{(n)}(E)$.
We will need to take into account that 
$F^0=(E/\Pl^{(n)}(E))^0=\pi i(E^0)$ and $i_F=\pi i_E\vert_{F^0}$.
The commutativity of the square below 
is contained in the proof of \cite[Theorem 2.4.12]{AAS}.
\[
\begin{tikzcd}[column sep=small]
 E^0 \arrow{r}{i_E}  
  & L_K(E) \arrow{d}{\pi} \\
  F^0\arrow{r}{i_F}\arrow[u,hook]   & L_K(F)
\end{tikzcd}
\]

\begin{proposition}
\label{aislu}
Let $L_K(E)$ and $L_K(F)$ be the Leavitt path algebras associated to the graphs $E$ and $F=E/\Pl^{(n)}(E)$, then $\Pl^{(n+1)}(E)=\{ v \in E^0 \; | \; v + I(\Pl^{(n)}(E)) \in \overline{\Pl(F)}^F\}$. 
\end{proposition}
\begin{proof}
All-through this proof we will shorten the notation
$\overline{\Pl(F)}^F$ to $\overline{\Pl(F)}$.
We know $\soc^{(n)}(L_K(E))=I(\Pl^{(n)}(E))$ and $$\soc(L_K(F))=\soc\left(\frac{L_K(E)}{\soc^{(n)}(L_K(E))}\right)= \frac{\soc^{(n+1)}(L_K(E))}{\soc^{(n)}(L_K(E))}.$$ For the first containment, consider $v \in \Pl^{(n+1)}(E)$. Thus $v + I(\Pl^{(n)}(E)) \in \soc(L_K(F))$ which implies $v+ I(\Pl^{(n)}(E)) \in I(\Pl(F)) \cap F^0$. Then $v + I(\Pl^{(n)}(E)) \in \overline{\Pl(F)}$.
For the converse, let $v$ be such that $v + I(\Pl^{(n)}(E)) \in \overline{\Pl(F)}=\cup_{i\ge 0}\Lambda^i$ (see \eqref{onion}).
Recall that $\Lambda^0=\Pl(F)$.
We prove that for any $i$, one has 
$$v+I(\Pl\nolimits^{(n)}(E))\in \Lambda^i\implies v\in \Pl\nolimits^{(n+1)}(E).$$
For $i=0$ we need to prove that if
$v+I(\Pl^{(n)}(E))\in \Pl(F)$ then $v\in \Pl^{(n+1)}(E)$. Take  $v + I(\Pl^{(n)}(E)) \in \soc(L_K(F))$. Then $v \in \soc^{(n+1)}(L_K(E)) \cap E^0 = I(\Pl^{(n+1)}(E))\cap E^0$, that is, $v\in \Pl^{(n+1)}(E)$.
Assume now that for some $i$ we have the implication:
$$v+I(\Pl\nolimits^{(n)}(E))\in \Lambda^i\implies v\in \Pl\nolimits^{(n+1)}(E).$$
Now we prove that 
$$v+I(\Pl\nolimits^{(n)}(E))\in \Lambda^{i+1}\implies v\in \Pl\nolimits^{(n+1)}(E).$$
So we consider $v+I(\Pl^{(n)}(E))\in\Lambda^{i+1}\setminus\Lambda^i$. We know that $v$ is a regular vertex and since we are considering the row-finite case, we have $r_F(s_F^{-1}(v))=\{w_1,\ldots,w_n\}$. 
Since $w_j+I(\Pl\nolimits^{(n)}(E))\in\Lambda^i$ applying the induction hypothesis we have that each $w_j\in\Pl^{(n+1)}(E)$
for $1\le j\le n$. On the other hand, we may have
$r_E({s_E}^{-1}(v))=\{w_1,\ldots,w_n,w_{n+1},\ldots w_k\}$. But then, for $j\ge 1$, one has $w_{n+j}\in \Pl^{(n)}(E)$ hence these elements are in $\Pl^{n+1}(E)$. 
In conclusion $w_j\in\Pl^{(n+1)}(E)$ for any index $j$.
So the CK2 applied to the vertex $v$ of $E$ gives $v\in\Pl^{(n+1)}(E)$.
\end{proof}

\begin{notation}{\rm
For two subsets $E_1^0, E_2^0$ of vertices of $E^0$ we write $E_1^0 \subseteq^{1} E_2^0$ if all the vertices of $E_1^0$ are contained in $E_2^0$ except at most one.}
\end{notation}

For the next result we will need to do a previous lemma.
\begin{lemma}\label{oportuno}
Let $L_K(E)$ be the Leavitt path algebra associated to a graph $E$. Then $\Pl^{(n)}(E)$ does not contain vertices that are base of a cycle in $E$.
\end{lemma}

\begin{proof}
 By induction on $n$, first it is clear for $n=1$ since $\Pl^{(1)}(E)=\overline{\Pl(E)}$. Suppose the condition holds for $\Pl^{(k)}(E)$ for $k < n$. Let $w \in \Pl^{(n)}(E)$ be such that it is a base of a cycle $c$ in $E$. By Proposition \ref{aislu}, we have that $w + I(\Pl^{(n-1)}(E)) \in \overline{\Pl(F)}^F$, where $F=E/\Pl^{(n-1)}(E)$. Write $\overline{\Pl(F)}^F=\cup_{i\ge 0}\Lambda^i$. If $w + I(\Pl^{(n-1)}(E)) \in \Lambda^0=\Pl(F)$ then $w \in \Pl^{(n-1)}(E)$ (because $c^0 \cap \Pl^{(n-1)}(E) \ne \emptyset$), but by induction hypothesis, $\Pl^{(n-1)}(E)$ does not contain vertices based at cycles so we get a contradiction. Next, we assume that $w \in \Lambda^i \setminus \Lambda^{i-1}$. So $r_F(s_F^{-1}(w)) \subseteq \Lambda^{i-1}$. We have two cases. First, imagine the cycle $c$ based at $w$ is such that $c \in  \path (F)$, then $w \in \Lambda^{i-1}$ which is not possible. So secondly, $c \notin \path (F)$, that is, there exists $u \in c^0$ with $u \in \Pl^{(n-1)}(E)$ hence $w \in \Pl^{(n-1)}(E)$, a contradiction.
\end{proof}

\begin{theorem}\label{encasa}
Let $L_K(E)$ be the Leavitt path algebra associated to a graph $E$. Then $\Pl^{(n+1)}(E)$ is the saturated closure of
\begin{equation}\label{chungon}
\{ v \in E^0 \; | \; T_E(v) \text{ is acyclic and }\forall w\in T_E(v), r_E(s_E^{-1}(w)) \subseteq^1 \Pl\nolimits^{(n)}(E)\}
\end{equation}
for $n  \ge 1$.
\end{theorem}
\begin{proof} It is straightforward to check that (\ref{chungon}) is a hereditary set. We denote $F:=E/\Pl^{(n)}(E)$. Consider $ v \in \Pl^{(n+1)}(E)$. Identify the vertices of $F$ with the corresponding vertices of $E$, i.e. $ v + I(\Pl^{(n)}(E))$ as a vertex of $F$ is identified with the vertex $v$ of $E$. Notice that $T_E(v)$ is acyclic by Lemma \ref{oportuno}. 

Next we prove that for any $w\in T_E(v)$, $r_E(s_E^{-1}(w))\subseteq^{1} \Pl^{(n)}(E)$. By Proposition \ref{aislu}, we have $v\in\overline{\Pl(F)}^F=\cup_{i\ge 0}\Lambda_i$. 
In case $v\in\Lambda^0=\Pl(F)$,
then $T_F(v)$ does not contain bifurcations of $F$.
So, take $w\in T_E(v)$, if $s^{-1}(w)=\{g_i\}_{i\in I}$ in $E$, the situation in $F$ is that either all the edges have disappeared when passing to $F$ or at most one, say $g_1$ survives. In this way, in the graph $F$ we have $\vert r_F(s_F^{-1}(w))\vert\le 1$. Whence $w$ is either a sink of $F$ or $r_F(s_F^{-1}(w))$ has cardinal $1$. This proves our claim for $i=0$. Now, assume that the property holds for any $\Lambda^k$ with $k \in\{0,1, \ldots, i\}$. Take $v\in\Lambda^{i+1}\setminus\Lambda^i$. Let
$r_F(s_F^{-1}(v))=\{w_1,\ldots,w_n\}$. Since $T_E(w_j)\subset T_E(v)$ then the tree of each $T_E(w_j)$ is acyclic, and since
$w_j\in \Lambda^i$, any vertex $v'\in T_E(w_j)$ satisfies 
 $r_E(s_E^{-1}(v'))\subseteq^{1} \Pl^{(n)}(E)$. Thus each $w_i$ for $i=1,\ldots,n$ is in the set \eqref{chungon}. We may have
$r_E({s_E}^{-1}(v))=\{w_1,\ldots,w_n,w_{n+1},\ldots w_q\}$. But then, for $j\geq 1$, one has $w_{n+j}\in \Pl^{(n)}(E)$. We know $T_E(w_{n+j})$ is acyclic and, on the other hand, for every $z \in T_E(w_{n+j})$ in fact $r_E(s_E^{-1}(z))\subseteq \Pl^{(n)}(E)$, so in particular $r_E(s_E^{-1}(z))\subseteq^{1} \Pl^{(n)}(E)$ hence $\{w_{n+1},\ldots w_q\}$ belongs to (\ref{chungon}). Finally applying CK2 to $v$ we have that $v$ is in the ideal generated by the set \eqref{chungon}. Applying \cite[Corollary 2.4.16 (i)]{AAS} we have that $v$ is in the saturated closure of the set in \eqref{chungon}.

To prove the converse it suffices to see that the set in \eqref{chungon} is contained in $\Pl^{(n+1)}(E)$. Assume that $v$ is a vertex such that $T_E(v)$ is acyclic and 
for any $w\in T_E(v)$ we have  $r(s^{-1}(w)) \subseteq^{1} \Pl^{(n)}(E)$. Then we have in $F$ that $s^{-1}(w)=\vacio$ or $\vert s^{-1}(w)\vert=1$. Thus $v$ is a line point of $F$ and taking into account Proposition \ref{aislu} we get $v\in\Pl^{(n+1)}(E)$.
\end{proof}

\begin{remark}

\label{aisld}\rm
Let $L_K(E)$ be the Leavitt path algebra associated to a graph $E$. Taking into account  Theorem \ref{encasa} for $n=2$, we have that $\Pl^{(2)}(E)$ is the saturated closure of
\begin{equation}\label{chungo}
\{ v \in E^0 \; | \; T_E(v) \text{ is acyclic and }\forall w\in T_E(v), r(s^{-1}(w)) \subseteq^1 \overline{\Pl(E)} \}.
\end{equation}

So, a Leavitt path algebra $L_K(E)$ verifies $\soc^{(2)}({L_K(E)})=L_K(E)$ if and only if $E^0={P_l}^{(2)}(E)$.
\end{remark}

\begin{example} \rm In order to illustrate Theorem \ref{encasa} for $n=2$ we compute the set $\Pl^{(2)}(E)$ for the following graph $E$:

\begin{equation} 
    \xymatrix{ E: & {\bullet}^{v_{1,1}} \ar[r]  & {\bullet}^{v_{1,2}} \ar[r]  & {\bullet}^{v_{1,3}} \ar[r]  &
               {\bullet}^{v_{1,4}} \ar@{.>}[r] &  \\
             & {\bullet}^{v_{2,1}} \ar[r] \ar[u] & {\bullet}^{v_{2,2}} \ar[r] \ar[u] & {\bullet}^{v_{2,3}} \ar[r] \ar[u] & {\bullet}^{v_{2,4}}\ar@{.>}[r] \ar[u] &  \\
             & {\bullet}^{v_{3,1}} \ar[r] \ar[u] & {\bullet}^{v_{3,2}} \ar[r] \ar[u] & {\bullet}^{v_{3,3}} \ar[r] \ar[u] & {\bullet}^{v_{3,4}}\ar@{.>}[r] \ar[u] &  \\
           }\label{despues}
           \end{equation}
\smallskip
In this case we have $\Pl^{(1)}(E)=\{v_{1,i}\}_{i \ge 1}$ and by Theorem \ref{encasa} we get the following equality  $\Pl^{(2)}(E)=\{v_{1,i}\}_{i \ge 1} \cup \{v_{2,j}\}_{j \ge 1}$. According to Proposition \ref{invariante}, the ideals generated by these sets are invariant under isomorphism and of course we have $\soc(L_K(E))=I(\{v_{1,i}\}_{i\ge 1})$ and $\soc^{(2)}(L_K(E))=I(\{v_{1,i}\}_{i \ge 1} \cup \{v_{2,j}\}_{j \ge 1})$.
Also the quotient graph $F = E/\Pl(E)$ is 

\begin{equation} 
    \xymatrix{ F: 
             & {\bullet}^{v_{2,1}} \ar[r]  & {\bullet}^{v_{2,2}} \ar[r] & {\bullet}^{v_{2,3}} \ar[r]  & {\bullet}^{v_{2,4}}\ar@{.>}[r]  &  \\
             & {\bullet}^{v_{3,1}} \ar[r] \ar[u] & {\bullet}^{v_{3,2}} \ar[r] \ar[u] & {\bullet}^{v_{3,3}} \ar[r] \ar[u] & {\bullet}^{v_{3,4}}\ar@{.>}[r] \ar[u] &  \\
           }\label{despues2}
           \end{equation}
\smallskip

\noindent so that $L_K(E)/\soc^{(2)}(L_K(E))=L_K(G)$ which is simple and coincides with its socle. This implies  $\soc^{(3)}(L_K(E))=L_K(E)$. Consequently $\Pl^{(3)}(E)=E^0$.
\begin{equation} 
    \xymatrix{ G:
             & {\bullet}^{v_{3,1}} \ar[r]  & {\bullet}^{v_{3,2}} \ar[r]  & {\bullet}^{v_{3,3}} \ar[r]  & {\bullet}^{v_{3,4}}\ar@{.>}[r]  &  \\
           }\label{despues3}
           \end{equation}
\smallskip
\end{example}

\begin{example}\rm Now this example shows the general case in Theorem \ref{encasa}. Let $E$ be the following graph, denoting $A:=L_K(E)$. We have $\soc^{(n)}(A)\subsetneq\soc^{(n+1)}(A)$ for any $n \in \N$.

\begin{equation} 
    \xymatrix{ E: & {\bullet}^{v_{1,1}} \ar[r]  & {\bullet}^{v_{1,2}} \ar[r]  & {\bullet}^{v_{1,3}} \ar[r]  &
               {\bullet}^{v_{1,4}} \ar@{.>}[r] &  \\
             & {\bullet}^{v_{2,1}} \ar[r] \ar[u] & {\bullet}^{v_{2,2}} \ar[r] \ar[u] & {\bullet}^{v_{2,3}} \ar[r] \ar[u] & {\bullet}^{v_{2,4}}\ar@{.>}[r] \ar[u] &  \\
             & {\bullet}^{v_{3,1}} \ar[r] \ar[u] & {\bullet}^{v_{3,2}} \ar[r] \ar[u] & {\bullet}^{v_{3,3}} \ar[r] \ar[u] & {\bullet}^{v_{3,4}}\ar@{.>}[r] \ar[u] &  \\
           & {\bullet}^{v_{n,1}} \ar[r] \ar@{.>}[u] & {\bullet}^{v_{n,2}} \ar[r] \ar@{.>}[u] & {\bullet}^{v_{n,3}} \ar[r] \ar@{.>}[u] & {\bullet}^{v_{n,4}}\ar@{.>}[r] \ar@{.>}[u] &  \\
           &   \ar@{.>}[u] &  \ar@{.>}[u] &  \ar@{.>}[u] &  \ar@{.>}[u] &  \\
           }\label{manana}
           \end{equation}

\medskip
In this case $\Pl^{(n)}(E)=\{v_{i,j} \; \vert \; i=1, \ldots ,n \text{ and }j\ge 1\}$ for $n \in \N \setminus \{0\}$. And we see that $\Pl^{(n)}(E)\subsetneq\Pl^{(n+1)}(E)$ for any $n$.
\end{example}

\section{The series of functors of a hereditary and saturated one}\label{rosa}

Let $H\colon\grph\to\set$ be a hereditary and saturated point functor. Fix a graph $E$ and
define $H^{(1)}:=H$. Assuming that $H^{(1)},\ldots, H^{(n)}$ are defined, then we define
$H^{(n+1)}$ applied to a graph $E$ as the hereditary and saturated subset of $E^0$ such that the ideal 
$I(H(E/ H^{(n)}(E)))$, which is an ideal in $L_K(E/ H^{(n)}(E))\cong L_K(E)/ I(H^{(n)}(E))$, satisfies 
\begin{equation}\tiny
    I\left(H\left(\frac{E}{H^{(n)}(E)}\right)\right)=\frac{I(H^{(n+1)}(E))}{I(H^{(n)}(E))}.
\end{equation}
\begin{remark}\label{gatoLeo}\rm Observe that, as a consequence of \cite[Corollary 2.9.11]{AAS} and \cite[Proposition 2.4.9]{AAS}, we have that graded ideals of a quotient algebra are quotient of graded ideals (though this fact seems to be more general and does not need the setting of Leavitt path algebras).
\end{remark}
By construction we have $H=H^{(1)}\subset H^{(2)}\subset\cdots\subset H^{(n)}\subset\cdots $
and each $H^{(n)}$ being hereditary and saturated.

\begin{proposition}\label{hinv} If $H$ is invariant under isomorphism, then the series of functors $H^{(n)}$ ($n\ge 1$) are invariant under isomorphism in the sense of definition \ref{hervor}.
\end{proposition}
\begin{proof}
Assume that $H^{(n)}$ is invariant under isomorphism. To prove that $H^{(n+1)}$ is also invariant, take any isomorphism $f\colon  L_K(E)\to L_K(F)$. Then it induces by passing to the quotient an isomorphism 
\begin{equation}\tiny\bar f\colon L_K\left(\frac{E}{H^{(n)}(E)}\right)\to L_K\left(\frac{F}{ H^{(n)}(F)}\right)
\end{equation}
and consequently $\bar f(I(H(E/ H^{(n)}(E))))= I(H(F/ H^{(n)}(F)))$. Thus 

\begin{equation}
 \tiny \bar f\left(\frac{I(H^{(n+1)}(E))}{I(H^{(n)}(E))}\right)=\frac{I(H^{(n+1)}(F))}{I(H^{(n)}(F))}
\Rightarrow f(I(H^{(n+1)}(E)))=I(H^{(n+1)}(F)).
\end{equation}
\end{proof}

The following result is the analogous to the one given in Theorem \ref{encasa} which was referred to the functor $\Pl^{(n+1)}$. Now we describe graphically $\Pc^{(n+1)}$.

\begin{theorem}\label{enbarco2}
Let $L_K(E)$ be the Leavitt path algebra associated to a graph $E$. Then $\Pc^{(n+1)}(E)$ is the hereditary and saturated closure of
\begin{equation}\label{och}
\mathfrak{S}_n= \{ v \in E^0 \; | \; v \in c^0, c \text{ is a cycle and }\forall f \text{ exit of } c, r(f) \in \Pc\nolimits^{(n)}(E)\}
\end{equation}
for $n  \ge 1$.
\end{theorem}
\begin{proof}
Firstly we prove the formula:
\begin{equation}\label{tutu}
     \Pc\nolimits^{(n)}(E) \subseteq \overline{\mathfrak{S}_n}^E\quad (\hbox{ in the sequel } \overline{\mathfrak{S}_n}^E \hbox{ will be shortened }\overline{\mathfrak{S}_n}).
\end{equation}
For $n=1$, it suffices to prove that $\Pc(E)\subset\mathfrak{S}_1$ which is trivial (because there is no exits in the cycles involved). Assume that $\Pc^{(k)}(E)\subset\overline{\mathfrak{S}_k}$ for $k<n$. Take $v\in\Pc^{(n)}(E)$ but 
$v\notin\Pc^{(n-1)}(E)$. Then 
$$v+I(\Pc\nolimits^{(n-1)}(E))\in \frac{I(\Pc^{(n)}(E))}{I(\Pc^{(n-1)}(E))}\cong 
I[\Pc(E/\Pc\nolimits^{(n-1)}(E))]$$
so that $v$ is in the hereditary and saturated closure of $\Pc(E/\Pc^{(n-1)}(E))$ which is (by \eqref{onion}) $\cup_{i\ge 0}\Lambda^i$ (closure in the quotient graph). We will prove by induction that each $\Lambda^i$ is contained in $\overline{\mathfrak{S}_n}$. 
If $v\in\Lambda^0$ then $v$ is in a cycle of $E$. If this cycle has no exits then $v \in \mathfrak{S}_n$. And if $v \in c^0$ and $c$ has an exit $f$  then since $\Pc\nolimits^{(n)}(E)$ is hereditary, $r(f) \in \Pc\nolimits^{(n)}(E)$. Consequently in this case 
$v\in\mathfrak{S}_n$. On the other hand, assume $\Lambda^j\subset\overline{\mathfrak{S}_n}$ 
for $j<i$. Take now $v\in\Lambda^i$ with $i>0$ (but $v\notin\Lambda^{i-1}$). Let $s_E^{-1}(v)=\{g_1,\ldots,g_r\}$, then we may assume (reordering if necessary) that for the graph $G:=E/\Pc^{(n-1)}(E)$, we have 
$s_G^{-1}(v)=\{g_1,\ldots, g_l\}$; while the others $\{g_{l+1},\ldots, g_r\}$ satisfy 
$r_E(g_i)\in \Pc^{(n-1)}(E)\subset \overline{\mathfrak{S}_{n-1}} \subset \overline{\mathfrak{S}_n}$. Note that 
$r_G(s_G^{-1}(v))\in\Lambda^{i-1}\subset\overline{\mathfrak{S}_n}$. So $r_E(s_E^{-1}(v))\subset\overline{\mathfrak{S}_n}$ hence $v\in\overline{\mathfrak{S}_n}$.
So far we have proved formula \eqref{tutu}.

Let us prove now that $\Pc\nolimits^{(n+1)}(E)$ is contained in the saturated closure of the set $\mathfrak{S}_n$.
If $v\in\Pc\nolimits^{(n+1)}(E)\setminus \Pc\nolimits^{(n)}(E)$ then $v+I(\Pc\nolimits^{(n)}(E))\in I(\Pc\nolimits^{(n+1)}(E))/I(\Pc\nolimits^{(n)}(E))\triangleleft L_K(E)/I(\Pc\nolimits^{(n)}(E))$. Let us denote 
$F:=E/\Pc\nolimits^{(n)}(E)$ the quotient graph.  Let $\theta\colon L_K(E)/I(\Pc\nolimits^{(n)}(E))\to L_K(F)$ be as explained in Remark \ref{tita Antonia}, that is, $\theta$ is the canonical isomorphism such that 
$v+I(\Pc\nolimits^{(n)}(E))\buildrel{\theta}\over{\mapsto} v$ (as element of the graph $F$). Restricting $\theta$, we have an isomorphism
$\theta\colon I(\Pc\nolimits^{(n+1)}(E))/I(\Pc\nolimits^{(n)}(E))\to I(\Pc^{(1)}(F))$. Consequently, given that $v+I(\Pc\nolimits^{(n)}(E))\in I(\Pc\nolimits^{(n+1)}(E))/I(\Pc\nolimits^{(n)}(E))$, applying
$\theta$ we have $v\in \Pc^{(1)}(F)$ (in the graph $F$). Write now $\Pc^{(1)}(F)=\cup_{i\ge 0}\Lambda^i(\Pc(F))$ (again \eqref{onion}) and let us prove by induction that 
\begin{equation}\label{kjurg}
T_F(v)\cap\Lambda^i\subset\overline{\mathfrak{S}_n}.
\end{equation}
For $i=0$ we must check that $T_F(v)\cap\Pc(F)\subset\overline{\mathfrak{S}_n}$. So, if $w\in T_F(v)\cap \Pc(F)$ then
$w\in c^0$ where the cycle has no exit in $F$ but has exits in $E$. If $f$ is any exit of $c$ then $r(f)\in E^0\setminus F^0$ hence $r(f)\in \Pc^{(n)}(E)$. Whence $w\in\mathfrak{S}_n$. Assuming $T_F(v)\cap\Lambda^i\subset \overline{\mathfrak{S}_n}$, we prove $T_F(v)\cap\Lambda^{i+1}\subset \overline{\mathfrak{S_n}}$: 
if $u\in T_F(v)\cap\Lambda^{i+1}$, then $r_F(s_F^{-1}(u))\in T_F(v)\cap \Lambda^{i}\subset \overline{\mathfrak{S}_n}$.
However, in order to conclude that $u\in\overline{\mathfrak{S}_n}$ we must see that 
$r_E(s_E^{-1}(u))\subset \overline{\mathfrak{S}_n}$. But $r_E(s_E^{-1}(u))\setminus r_F(s_F^{-1}(u))\subset \Pc\nolimits^{(n)}(E)\subset \overline{\mathfrak{S}_n}$ by \eqref{tutu}. Thus $r_E(s_E^{-1}(u))\in \overline{\mathfrak{S}_n}$, so  $u\in\overline{\mathfrak{S}_n}$. This completes the induction proof of formula 
\eqref{kjurg}. Now, since $T_F(v)\subset \overline{\Pc(F)}=\cup_{i\ge 0}\Lambda^i(\Pc(F))$, we have  
$T_F(v)=T_F(v)\cap\overline{\Pc(F)}=\cup_{i\ge 0}(T_F(v)\cap \Lambda^i)\subset\overline{\mathfrak{S}_n}$.
Consequently  $T_F(v)\subset \overline{\mathfrak{S}_n}$ implying $v\in\overline{\mathfrak{S}_n}$.

For the converse relation it suffices to see that $\mathfrak{S}_n$ is contained in $\Pc^{(n+1)}(E)$.
So consider $v$ a vertex in $\mathfrak{S}_n$. If $v$ is in a cycle without exits of $E$ then $v\in\Pc^{(1)}(E)\subset\Pc^{(n+1)}(E)$ so we are done (note that the sets $\Pc^{(n)}(E)$ form an ascending chain:  $\Pc^{(k)}(E)\subseteq \Pc^{(k+1)}(E)$ for any $k$).
If $v$ is in a cycle with exits $c$ of $E$, then for any exit $f$ of $c$ we know $r_E(f)\in\Pc^{(n)}(E)$. Thus, relative to the graph $F:=E/\Pc^{(n)}(E)$ we have that
the exit $f$ is not an edge of $F$ because its target is not in $F^0$. So the cycle $c$ has no exit in $F$. Whence 
$$v\in I(\Pc\nolimits^{(n)}(F))=\theta\left(\frac{I(\Pc^{(n+1)}(E))}{I(\Pc^{(n)}(E))}\right)$$ 
hence $v=\theta(w+I(\Pc^{(n)}(E)))$ for some $w\in I(\Pc^{(n+1)}(E))$. But $\theta(w+I(\Pc^{(n)}(E)))=w$ which implies $v=w\in I(\Pc^{(n+1)}(E)) \cap E^0$, i.e., $v \in \Pc^{(n+1)}(E)$.
\end{proof}

\begin{example} \rm Now we compute the sets $\Pc\nolimits^{(n)}(E)$ for the following graph  in order to illustrate Theorem \ref{enbarco2}. Actually we have that $\Pc^{(n)}(E)= \{v_1, \ldots, v_n\}$ for $n \ge 1$. Also observe that $\Pc^{(n})(E) \subsetneq \Pc^{(n+1)}(E)$ for any $n$.

\begin{equation} 
   E:   \xymatrix{
    \ar@{.>}[r]  & {\bullet}_{v_{4}}\ar@(ul,ur) \ar[r]  & {\bullet}_{v_{3}} \ar[r] \ar@(ul,ur) & {\bullet}_{v_{2}} \ar[r] \ar@(ul,ur) & {\bullet}_{v_{1}}\ar@(ur,dr) &  \\
           }\label{conica}
           \end{equation}
\smallskip
\end{example}

\subsection{Mixed point-functors}\label{colada}
Finally, we define a kind of composition of functors which gives new hereditary and saturated functors when it is applied to hereditary and saturated ones. Furthermore, if the starting functors are invariant, then the composite is also invariant. 
Assume that for $i=1,2$ we have hereditary and saturated point functors $H_i$. Then we can construct a new point-functor $H_2* H_1$ by the following procedure: for any graph $E$ consider the ideal
$I[H_2(E/H_1(E))]\triangleleft L_K(E/H_1(E))\buildrel{\theta}\over\cong L_K(E)/I(H_1(E))$. Thus there exists a unique hereditary and saturated subset of $E$ (see Remark \ref{gatoLeo}), denoted $(H_2\ast H_1)(E)$, such that 
\begin{equation}
I[(H_2\ast H_1)(E)]/I(H_1(E))
\buildrel{\theta}\over{\cong}
I[H_2(E/H_1(E))]. 
\end{equation}

\begin{example}{\rm For instance, in the graph $E$, below we can consider the functors $\overline{\Pl}$ and $\overline{\Pc}$.
\[
\xygraph{!{(0.25,0)}*+{E\colon}
!{<0cm,0cm>;<1cm,0cm>:<0cm,1cm>::}
!{(2,0)}*+{\bullet}="d"
!{(3,0)}*+{\bullet}="g"
!{(3,-0.25)}*+{\hbox{\tiny $v$}}="v"
!{(2,-0.25)}*+{\hbox{\tiny $u$}}="u"
!{(4,0)}*+{\bullet}="h"
!{(4,-0.25)}*+{\hbox{\tiny $w$}}="e"
"d" :@(ul,ur) "d"
"d":"g" 
"g":"h"
"g" :@(ul,ur) "g"
}\]
Then $\overline{\Pl}(E)=\{w\}$ and $\overline{\Pc}(E)=\vacio$. However
$(\overline{\Pc}\ast \overline{\Pl})(E)=\{v,w\}$ and 
$(\overline{\Pc}\ast (\overline{\Pc}\ast \overline{\Pl}))(E)=E^0$.
In the example, we see that in general we do not have commutativity of the operation $*$ because $(\overline{\Pl}\ast \overline{\Pc})(E)=\{w\}$.
}
\end{example}

The empty set point functor $\vacio\colon\grph\to\set$ mapping any graph to the emptyset is an identity element for the $*$-operation:
$$\frac{I((H*\vacio)(E))}{I(\vacio(E))}{\buildrel{\hbox{\tiny def}}\over\cong} I(H(E/\vacio(E)))\cong I(H(E))$$
whence $H*\vacio\cong H$. On the other hand 
$$\frac{I((\vacio* H)(E))}{I(H(E))}{\buildrel{\hbox{\tiny def}}\over\cong} I(\vacio(E/H(E)))=0$$
implying $\vacio*H\cong H$.
Also, it is remarkable that:
\begin{proposition}
If $H_i$  are invariant hereditary and saturated point functors for $i\in \{1,2\}$, then $H_2\ast H_1$ is also invariant.
\end{proposition}
\begin{proof}
We know that for any isomorphism $f\colon L_K(E)\to L_K(E')$ one has $f(I(H_i(E)))=I(H_i(E'))$, ($i=1,2$).
Consider the induced isomorphisms 

$$\begin{matrix}
\omega_i\colon L_K(E)/I(H_i(E))\cong L_K(E')/I(H_i(E'))\cr 
{ x+I(H_i(E))\buildrel{\omega_i}\over {\mapsto} f(x)+I(H_i(E'))}
\end{matrix}$$

together with the canonical isomorphisms $$\theta_i\colon L_K(E)/I(H_i(E))\to L_K(E/H_i(E)),\quad \theta_i'\colon L_K(E')/I(H_i(E'))\to L_K(E'/H_i(E')).$$ Then define $\bar f_i$ as the unique isomorphism
$\bar f_i\colon L_K(E/H_i(E))\to L_K(E'/H_i(E'))$ making commutative the diagram
\begin{center}
\begin{tikzcd}
\frac{L_K(E)}{I(H_
(E))} \arrow[r, "\omega_i"] \arrow[d, "\theta_i"']
& \frac{L_K(E')}{I(H_i(E'))} \arrow[d, "\theta_i'"] \\
L_K(\frac{E}{H_i(E)}) \arrow[r, "\bar f_i"]
& L_K(\frac{E'}{H_i(E')})
\end{tikzcd}
\end{center}
Then $\theta_i$ restricts to an isomorphism $\theta_i\colon I((H_j\ast H_i)(E))/I(H_i(E))\cong I(H_j(E/H_i(E)))$ 
so that $\bar f_i\theta_i$ is an isomorphism and
{\tiny$$  \bar f_i \theta_i \left[I((H_j\ast H_i)(E))/I(H_i(E))\right]\cong \bar f_i[I(H_j(E/H_i(E)))]=I(H_j(E'/H_i(E')))\cong I((H_j\ast H_i)(E'))/I(H_i(E'))$$}
implying that 
$f[I((H_j\ast H_i)(E))]=I((H_j\ast H_i)(E'))$.
\end{proof}

The $*$ operation has a kind of associativity property which can be formalized in terms of natural isomorphism of functors:

\begin{proposition}\label{associative}
Let $H_i$  be hereditary and saturated invariant point functors for $i\in \{1,2,3\}$, then there is a natural isomorphism of functors $(H_1*H_2)*H_3\cong H_1*(H_2*H_3)$.
\end{proposition}
\begin{proof}
{\tiny $$L_K\left(\frac{E}{H_2*H_3(E)}\right)\cong\frac{L_K(E)}{I(H_2*H_3(E))}\cong\frac{L_K(E)/I(H_3(E))}{I(H_2*H_3(E))/I(H_3(E))}\cong \frac{L_K(E/H_3(E))}{I(H_2(E/H_3(E)))}\cong
L_K\left(\frac{E/H_3(E)}{H_2(E/H_3(E))}\right).$$}
Let $f$ be the isomorphism from $L_K\left(\frac{E}{H_2*H_3(E)}\right)$ to $L_K\left(\frac{E/H_3(E)}{H_2(E/H_3(E))}\right)$, applying that $H_1$ is invariant, we have  
{\tiny $$ \frac{I((H_1*(H_2*H_3))(E)}{I(H_2*H_3(E))}\cong I\left(H_1\left(\frac{E}{H_2*H_3(E)}\right)\right){\buildrel{f}\over{\cong}}
\ I\left(H_1\left(\frac{E/H_3(E)}{H_2(E/H_3(E))}\right)\right)\cong \frac{I((H_1*H_2)(E/H_3(E))}{I(H_2(E/H_3(E))}\cong$$ $$\frac{I(((H_1*H_2)*H_3))(E)}{I(H_2*H_3(E))}.$$}
But then 
 $I((H_1*(H_2*H_3))(E))\cong I(((H_1*H_2)*H_3)(E))$ which induces the natural isomorphism of functors $(H_1*H_2)*H_3\cong H_1*(H_2*H_3)$.
\end{proof}


\begin{thebibliography}{10}

\bibitem{AALP}\textsc{Gene Abrams, Pham N. Anh, Adel Louly, Enrique Pardo}, \textit{The classification question for Leavitt path algebras}, J. Algebra {\bf 320} (2008), 1983–2026.


\bibitem{AAS} \textsc{Gene Abrams, Pere Ara, Mercedes Siles Molina}, Leavitt path algebras. Lecture Notes in Mathematics 2191, Springer (2017).

\bibitem{ARS} \textsc{Gene Abrams, Kulumani Rangaswamy,  Mercedes Siles Molina}, \textit{The socle series of a Leavitt path algebra}. Isr. J. Math. \textbf{184}  (2011), 413--435.











\bibitem{ABS} \textsc{Gonzalo Aranda Pino,  Jose Brox, Mercedes Siles Molina}, \textit{Cycles in Leavitt path algebras by means of idempotents}. Forum Math. \textbf{27}  (2015), 601--633.





\bibitem{AMMS1} \textsc{Gonzalo Aranda Pino, Dolores Mart\'in Barquero, C\'andido Mart\'in Gonz\'alez, Mercedes Siles Molina}, \textit{The socle of a Leavitt path algebra}. J. Pure Appl. Algebra \textbf{212} (2008), 500--509.

\bibitem{AMMS2} \textsc{Gonzalo Aranda Pino, Dolores Mart\'in Barquero, C\'andido Mart\'in Gonz\'alez, Mercedes Siles Molina}, \textit{Socle theory for Leavitt path algebras of arbitrary graphs}. Rev. Mat. Iberoam. \textbf{26} (2) (2010), 611--638.

 \bibitem{CGKS} \textsc{Vural Cam, Cristóbal Gil Canto, Muge Kanuni,  Mercedes Siles Molina}, \textit{Largest ideals in Leavitt path algebras}. Mediterr. J. Math.(2020), 17:66. 

 \bibitem{CMMS19} \textsc{Lisa O. Clark, Dolores Mart\'in Barquero, C\'andido Mart\'in Gonz\'alez; Mercedes Siles Molina}, \textit{Using the Steinberg algebra model to determine the center of any Leavitt path algebra}. Israel J. Math. {\bf 230} (2019), no. 1, 23–44.

\bibitem{CMMS}\textsc{Lisa O. Clark, Dolores Mart\'{\i}n Barquero, C\'andido Mart\'{\i}n Gonz\'alez, Mercedes Siles Molina}, \textit{ Using Steinberg algebras to study decomposability of Leavitt path algebras}. Forum Math. {\bf 6} (29) (2017), 1311--1324.

\bibitem{CMMSS} \textsc{Mar\'ia G. Corrales Garc\'\i a,  Dolores Mart\'{\i}n Barquero, C\'andido Mart\'{\i}n Gonz\'{a}lez, Mercedes Siles Molina, Jos\'e F. Solanilla Hern\'andez}, \textit{Extreme cycles. The center of a Leavitt path algebra}. Pub. Mat. \textbf{60} (2016), 235--263.

\bibitem{DU}\textsc{James Dugundji}, \text{ Topology}. Allyn and Bacon, Inc. (1966).

\bibitem{DanielDanilo}\textsc{Daniel Goncalves, Danilo Royer}, \textit{A note on the regular ideals of Leavitt path algebras}. Preprint. https://arxiv.org/pdf/2006.03634.pdf

 \bibitem{Hamana} \textsc{Masamichi Hamana}, \textit{The centre of the regular monotone completion of a $C^*$-algebra}. J. London Math. Soc. (2) {\bf 26} (3) (1982), 522--530.





\bibitem{R} \textsc{Kulumani Rangaswamy}, \textit{The multiplicative ideal theory of Leavitt path algebras}. J. Algebra \textbf{487} (2017), 173--199.




\end{thebibliography}
\end{document}